\documentclass[11pt]{amsart} 

\usepackage{latexsym}
\usepackage{amsfonts}
\usepackage{amsmath,amsthm,amssymb}

\usepackage{fullpage}   
\usepackage{mathabx}
\usepackage{tikz}
\usetikzlibrary{matrix}
\usepackage[colorlinks=true, linkcolor=blue, citecolor=magenta]{hyperref}
\usepackage{mathrsfs}
\usepackage{soul}
\usepackage{arydshln}

\newcommand{\cal}{\mathcal}

\def\epsilon{\varepsilon}
\def\phi{\varphi}

\def\hat{\widehat}

\def\subset{\subseteq}

\newcommand{\Supp}{\mbox{Supp}}
\newcommand{\supp}{\mbox{Supp}}
\newcommand{\Curr}{\mbox{Curr}}

\newcommand{\Aut}{\mbox{Aut}}

\newcommand{\card}{\mbox{card}}

\newcommand{\FN}{F_N}   

\newcommand{\R}{\mathbb R}
\newcommand{\Z}{\mathbb Z}

\newcommand{\N}{\mathbb N}

\newcommand{\Per}{\text{\rm Per}}

\def\strutdepth{\dp\strutbox}
\def \ss{\strut\vadjust{\kern-\strutdepth \sss}}
\def \sss{\vtop to \strutdepth{
\baselineskip\strutdepth\vss\llap{$\diamondsuit\;\;$}\null}}

\def\strutdepth{\dp\strutbox}
\def \sst{\strut\vadjust{\kern-\strutdepth \ssss}}
\def \ssss{\vtop to \strutdepth{
\baselineskip\strutdepth\vss\llap{$\spadesuit\;\;$}\null}}

\def\strutdepth{\dp\strutbox}
\def \ssh{\strut\vadjust{\kern-\strutdepth \sssh}}
\def \sssh{\vtop to \strutdepth{
\baselineskip\strutdepth\vss\llap{$\heartsuit\;\;$}\null}}

\def\qed{\hfill\rlap{$\sqcup$}$\sqcap$\par}
\def\bar{\overline}

\def\strutdepth{\dp\strutbox}
\def \ss{\strut\vadjust{\kern-\strutdepth \sss}}
\def \sss{\vtop to \strutdepth{
\baselineskip\strutdepth\vss\llap{$\diamondsuit\;\;$}\null}}

\def\strutdepth{\dp\strutbox}
\def \sst{\strut\vadjust{\kern-\strutdepth \ssss}}
\def \ssss{\vtop to \strutdepth{
\baselineskip\strutdepth\vss\llap{$\spadesuit\;\;$}\null}}

\def\qed{\hfill\rlap{$\sqcup$}$\sqcap$\par}

\newtheorem{thm}{Theorem}[section]

\newtheorem{lem}[thm]{Lemma}
\newtheorem{prop}[thm]{Proposition}

\theoremstyle{definition}
\newtheorem{defn}[thm]{Definition}
\newtheorem{example}[thm]{Example}

\newtheorem{rem}[thm]{Remark}
\newtheorem{defn-rem}[thm]{Definition-Remark}
\newtheorem{convention}[thm]{Convention}

\newtheorem{warning}[thm]{Warning}

\theoremstyle{remark}

\numberwithin{equation}{section}

\begin{document}

\author[N.~Bedaride]{Nicolas B\'edaride} 
\author[A.~Hilion]{Arnaud Hilion}
\author[M.~Lustig]{Martin Lustig}

\address{\tt 
Aix Marseille Universit\'e, CNRS, 
I2M UMR 7373,
13453  Marseille, 
France}

\email{\tt nicolas.bedaride@univ-amu.fr}

\address{\tt
Institut de Math\'ematiques de Toulouse, UMR 5219, Universit\'e de Toulouse, CNRS, UPS, F-31062 Toulouse Cedex~9, France}

\email{\tt arnaud.hilion@math.univ-toulouse.fr}

\address{\tt 
Aix Marseille Universit\'e, CNRS, 
I2M UMR 7373,
13453  Marseille, 
France}

\email{\tt martinlustig@gmx.de}

\title[The measures transfer]
{The measure transfer for 
subshifts 
induced by a morphism of free monoids}
 
\begin{abstract} 
Every non-erasing monoid morphism $\sigma: \cal A^* \to \cal B^*$ induces a {\em measure transfer map} $\sigma_X^\cal M: \cal M(X) \to \cal M(\sigma(X))$ between 
the 
measure cones $\cal M(X)$ and $\cal M(\sigma(X))$, associated to any subshift $X \subset \cal A^\Z$ and its image subshift $\sigma(X) \subset \cal B^\Z$ respectively. We 
define and study 
this map in detail and show that it is continuous, linear and functorial.  It also turns out to be surjective 
\cite{BHL2.8-II}. Furthermore, an efficient 
technique to compute 
the value of the transferred measure $\sigma_X^\cal M(\mu)$ on any cylinder $[w]$ (for $w \in \cal B^*$) is presented.

\smallskip
\noindent
{\bf Theorem:}
If a non-erasing morphism $\sigma: \cal A^* \to \cal B^*$ is injective on the shift-orbits of some subshift $X \subset \cal A^\Z$, then $\sigma^\cal M_X$ is injective.

\smallskip

The assumption on $\sigma$ that it is  ``injective on the shift-orbits of $X$'' is strictly weaker than ``recognizable in $X$'', and strictly stronger than ``recognizable for aperiodic points in $X$''. The last assumption does in general not suffice to obtain the injectivity of the measure transfer map $\sigma_X^\cal M$.

\end{abstract} 
 
\thanks{The first author was partially supported by 
ANR Project IZES ANR-22-CE40-0011}

\subjclass[2010]{Primary 37B10, Secondary 37A25, 37E25}
 
\keywords{subshift, invariant measure, recognizable monoid morphism}
 
\maketitle

\section{Introduction}
\label{sec:intro}

The prime objects of this paper are morphisms $\sigma: \cal A^* \to \cal B^*$ of free monoids $\cal A^*$ and $\cal B^*$ over finite sets $\cal A$ and $\cal B$ respectively. These sets are called {\em alphabets}, and their elements are {\em letters}, denoted here by $a_k \in \cal A$ or $b_j \in \cal B$. To $\cal A$ and $\cal B$ there are canonically associated the spaces $\cal A^\Z$ and $\cal B^\Z$, equipped both with the product topology and a {\em shift operator}, here always denoted by $T$. All morphisms $\sigma$ in this paper are {\em non-erasing} (i.e. 
none of the $a_k \in \cal A$ are mapped to the empty word), 
and such $\sigma$ induces canonically a map $\sigma^\Z: \cal A^\Z \to \cal B^\Z$. But while $\sigma^\Z$ is continuous, in general it fails to be a ``morphism of dynamical systems'' from $(\cal A^\Z, T)$ to $(\cal B^\Z, T)$ in the classical topological dynamics meaning, in that most of the time we have
\begin{equation}
\label{eq1.1}
T \circ \sigma^\Z \,\, \neq \,\, \sigma^\Z \circ T \, .
\end{equation}
This creates a number of well known technical problems; for instance any non-empty, closed, shift-invariant subset (called a {\em subshift}) $X \subset \cal A^\Z$ has $\sigma^\Z$-image that is in general not shift-invariant, and thus not a subshift by itself. Still, there is a well defined image subshift (see Definition-Remark \ref{image-shift} below), denoted here by $\sigma(X)$, which has been studied previously in many occasions. 

\medskip

The focus of this paper is on the set $\cal M(\cal A^\Z)$ of measures $\mu$ on $\cal A^\Z$ that are {\em invariant}, which means that 
$\mu$ is a 
finite 
Borel measure, and that the shift operator preserves $\mu\,$. The map $\sigma^\Z$ gives us directly the classical ``push-forward'' measure $\mu_*$ on $\cal B^\Z$, but because of (\ref{eq1.1}) this measure will almost never be invariant, as it will not be preserved by the shift operator on $\cal B^\Z$. Nevertheless, any invariant measure $\mu$ on $\cal A^\Z$ can canonically be ``transferred'' by $\sigma$ to define an invariant measure on $\cal B^\Z$. 
Setting up properly and studying 
carefully this 
{\em measure transfer map}
$$\sigma^\cal M: \cal M(\cal A^\Z) \to \cal M(\cal B^\Z)$$
induced by any non-erasing free monoid morphism $\sigma: \cal A^* \to \cal B^*$ is the main purpose of this paper. 
We show (see Subsection \ref{sec:3.4} and Lemma \ref{ergodic-new}):

\begin{prop}
\label{1.3}
Let $\sigma: \cal A^* \to \cal B^*$ be any non-erasing morphism, and let $X \subset \cal A^\Z$ be any subshift over $\cal A$, with image subshift $Y := \sigma(X) \subset \cal B^\Z$. Then the induced measure transfer map $\sigma^\cal M$ restricts/co-restricts to a well defined map
$$\sigma_X^\cal M: \cal M(X) \to \cal M(Y) 
\, , \,\, \mu \mapsto \mu^\sigma$$
which has the following properties:
\begin{enumerate}
\item
$\sigma_X^\cal M$ is an $\R_{\geq 0}$-linear map of cones.
\item
$\sigma_X^\cal M$ is continuous.
\item
$\sigma_X^\cal M$ is functorial: $(\sigma' \circ \sigma)_X^\cal M = {\sigma'}_{\sigma(X)}^\cal M \circ \sigma_X^\cal M$ 
\item
If the subshift $X' \subset X$ is the support of 
a measure $\mu \in \cal M(X)$, then $\sigma(X')$ is the support of $\sigma_X^\cal M(\mu)$.
\item
If a measure $\mu \in \cal M(X)$ is ergodic, then so is $\mu^\sigma \in \cal M(Y)$.
\end{enumerate}
(In addition, the map $\sigma_X^\cal M$ is surjective, see Proposition 4.4 of \cite{BHL2.8-II}.)
\end{prop}

As indicated already before, the above introduced measure $\mu^\sigma$, obtained from a given measure $\mu \in \cal M(X)$ via the measure transfer induced by a morphism $\sigma$, is almost always quite different from the push-forward measure $\mu_*$ as defined by $\sigma$. It turns out that the popular sources for basics in symbolic dynamics do only consider this very particular case $\mu^\sigma = \mu_*\,$, see for instance \cite{Queff}, Proposition 5.22\,. The general set-up for the measure transfer, as studied here, doesn't seem to be available in the existing literature. The basic properties of the transferred measure proved below are in particular used in our cousin paper \cite{BHL2.8-II} for the purpose of an explicit treatment of the transferred measure by means of S-adic expansions of a given subshift.

In order to underline that, despite its divergence from the usual push-forward concept, the measure transfer is a very natural tool under the given circumstances, we'd like to mention that there are several ``ad hoc'' occasions where the measure transfer has already appeared beforehand, see Example \ref{1.2} below. A special but rather relevant such case takes place if the given morphism is an endomorphisms $\sigma: \cal A^* \to \cal A^*$, thus called a {\em  substitution}. We then have the symbolic dynamics version of the well known and important situation where a system of zero entropy coexists with a system of positive entropy, giving rise to a wealth of subtle and profound interactions. For geometers, a geometric counterpart is obtained, for example, by considering the attractive lamination of a pseudo-Anosov homeomorphism $f$ on a hyperbolic surface: this is a geodesic lamination of zero entropy (for the geodesic flow), but the action of the pseudo-Anosov is of positive entropy (given by the dilation coefficient $\lambda$ of this homeomorphism). In such contexts, it is natural to consider the invariant measures of the zero-entropy system that are {\em projectively} invariant under the positive entropy transformation operator (the substitution $\sigma$ or the pseudo-Anosov $f$ in the above examples). Relating back to our main topic, the measure transfer, we observe:

\begin{rem}
\label{intro-new}
Any non-erasing substitution $\sigma: \cal A^* \to \cal A^*$ with $\sigma^n(\cal A) \cap \cal A = \emptyset$ for some sufficiently large $n \geq 1$
defines a {\em substitutive subshift} $X_\sigma$ which has the property that its image subshift $\sigma(X_\sigma)$ is equal to $X_\sigma\,$.
To every non-negative eigenvector $\vec v$ of the transition matrix $M(\sigma)$, with eigenvalue $\lambda > 1$, there is canonically associated an invariant measure $\mu_{\vec v}$ on the subshift $X_\sigma$ (see \cite{BHL2}). It can be shown (using Proposition 4.3 and Remark 4.2 of \cite{BHL2.8-II}) that the measure transfer map induced by $\sigma$ acts on $\mu_{\vec v}$ as homothety with stretching factor $\lambda\,$, giving
$$\mu_{\vec v}^\sigma = \lambda \cdot \mu_{\vec v}\,  .$$
\end{rem}

The last equality indicates already that 
the transferred measure of a probability measure will in general not be probability. Indeed, we have
(see Remark \ref{letter-image-frequency} and Proposition \ref{22.7} below):

\begin{prop}
\label{non-proba-image}
For any invariant measure $\mu$ on the full shift $\cal A^\Z$ and any non-erasing morphism $\sigma: \cal A^* \to \cal B^*$ the transferred measure 
$\mu^\sigma = \sigma^\cal M(\mu)$ satisfies 
$$
\mu^\sigma(\cal B^\Z) = \sum_{a_k \in \cal A} \sum_{b_j \in \cal B} |\sigma(a_k)|_{b_j} \cdot \mu([a_k]) =  \sum_{a_k \in \cal A} |\sigma(a_k)| \cdot \mu([a_k]) \, .$$
More specifically, for the ``vectors'' of letter frequencies we have the 
matrix equality
$$([\mu^\sigma(b_j)])_{b_j \in \cal B} \,\, = \,\, M(\sigma) \cdot ([\mu(a_k)])_{a_k \in \cal A} \, .$$
\end{prop}

Here $|w|_u$ denotes the number of (possibly overlapping) occurrences of the word $u$ as factor in the word $w$, and $[w]$ is the ``cylinder'' that consists of all biinfinite words $\ldots x_n x_{n+1} \ldots$ for which the 
positive half-word 
$x_1 x_2 \ldots$ starts with $w$ as prefix. By $M(\sigma)$ we denote the {\em incidence matrix} of 
$\sigma$, 
which has 
coefficient $|\sigma(a_k)|_{b_j}$ at the position $(j, k)$.

In order to compute the 
transferred 
measure $\mu^\sigma([w'])$ of 
an arbitrary 
cylinder $[w'] \subset \cal B^\Z$, as done in Proposition \ref{non-proba-image} for the special case $w' = b_j \in \cal B$, it turns out 
that it is useful to introduce the number $\lfloor\sigma(w) \rfloor_{u}$  of {\em essential occurrences} of $u$ as a factor of $\sigma(w)$, by which we mean that the first letter of $u$ occurs in the $\sigma$-image of first letter of $w$, and the last letter of $u$ occurs in the $\sigma$-image of last letter of $w$. The following is proved in Proposition \ref{3.5.5} below, and several examples of concrete computations of cylinder values $\mu^\sigma([w'])$ are given in Sections \ref{sec:image-measure} and \ref{sec:4d}.

\begin{prop}
\label{0.formula}
Let $\sigma: \cal A^* \to \cal B^*$ be any non-erasing monoid morphism, and let $\mu$ be any invariant measure on $\cal A^\Z$. Then for any $w' \in \cal B^*$ 
with $|w'| \geq 2$ 
the transferred measure $\mu^\sigma := \sigma^\cal M(\mu)$, evaluated on the cylinder $[w']$, 
is given by the finite sum
$$
\mu^\sigma([w']) = \sum_{w \,\in\, W(w')}
{\lfloor\sigma(w) \rfloor}_{w'} \cdot \mu([w])\, ,$$
for $W(w') = {\big \{}w \in \cal A^* \,: 
\, |w| \leq \frac{|w'| - 2}{\min\{|\sigma(a_k)| \,\,:\,\, a_k \in \cal A\}}+2{\big \}}$.
\end{prop}

\medskip

As stated above in Proposition \ref{1.3}, any invariant measure $\mu'$ on 
the image subshift $Y = \sigma(X)$ is equal to the transfer $\sigma_X^\cal M(\mu)$ of some measure $\mu$ on 
the given preimage subshift $X$. However, for a general non-erasing morphism $\sigma$, this measure $\mu$ will be far from uniquely determined by $\mu'$. 

On the other hand, the injectivity of the measure transfer map, if given, is a strong and useful tool in many circumstances, see Example \ref{6.9new}. 
In particular, it can be used as key ingredient for the construction of subshifts with entropy 0 but infinitely many distinct ergodic probability measures. In \S 7 of our cousin paper \cite{BHL2.8-II} this has been detailed out for the special case of minimal subshifts, so that the additional assumption that $\sigma$ is recognizable in $X$ could be used.
This additional assumption implies indeed the injectivity of the measure transfer map $\sigma_X^\cal M$ (see 
Corollary 3.9 of \cite{BHL2.8-II}). We show here the following stronger result (see Theorem \ref{5n.3} below):

\begin{thm}
\label{1n.inj}
Let $\sigma: \cal A^* \to \cal B^*$ be a non-erasing morphism of free monoids on finite alphabets, and let $X \subset \cal A^\Z$ be any subshift. 

If the map induced by $\sigma$ on the shift-orbits of $X$ is injective, 
then the measure transfer map $\sigma_X^\cal M: \cal M(X) \to \cal M(\sigma(X))$ is injective.
\end{thm}

To be specific, we'd like to note that the hypothesis ``injectivity on the set of shift-orbits of $X$'' is strictly weaker than ``recognizable in $X$'', and strictly stronger than ``recognizable for aperiodic points in $X$". Also, the converse to the conclusion of Theorem \ref{1n.inj} does not hold in full generality. Furthermore, the two properties ``recognizable for aperiodic points in $X$" and ``injectivity of the induced measure transfer map'' are logically independent. The implications among all of these properties and their refusals are conveniently summarized in Fig. 1, and an exhaustive discussion with proofs and counter-examples is given in Section \ref{sec:5d}.

\medskip

To terminate this introduction, we'd like to list 
some 
incidents where our readers may already have encountered the measure transfer map, perhaps ``disguised'' in a different setting or language:

\begin{example}
\label{1.2}
(1)
Let $S$ be a compact surface with non-empty boundary, and $\tau \subset S$ a train track that fills $S$ and satisfies the usual conditions on its complementary components (see 
for instance \cite{Kap09} for details). Then a proper choice of arcs $\alpha_1, \ldots, \alpha_d$ transverse to $\tau$ gives rise to intervals with the following property: Any oriented geodesic lamination $\Lambda \subset S$, for which we assume that it can be isotoped into an interval-fibered neighborhood $\cal N(\tau)$ of $\tau$ in such a way that $\Lambda$ becomes  transverse to all interval fibers, defines an interval exchange transformation system (IET) on the intervals $\alpha_k$, and thus a subshift $X_\Lambda \subset \cal A^\Z$ for  $\cal A = \{\alpha_1, \ldots, \alpha_d\}$. Any transverse measure $\mu_\Lambda$ on $\Lambda$ defines an invariant measure $\mu$ on $X_\Lambda$.

Assume now that (as shown by Thurston for any ``pseudo-Anosov'' homeomorphism) that some homeomorphism $h: S \to S$, after being properly isotoped, maps $\cal N(\tau)$ in an interval-fiber preserving fashion into $\cal N(\tau)$. Then any transverse measure $\mu_\Lambda$ gives rise to an ``$h$-image transverse measure'' $\mu'_\Lambda$, which, when translated back to the corresponding invariant measure $\mu'$ on 
$X_{h(\Lambda)}$, 
turns out to be precisely the transferred measure $\sigma^\cal M(\mu)$, where $\sigma$ is the morphism on $\cal A^*$ 
induced by $h$.

In the most frequently considered pseudo-Anosov case the 
$h$-invariant 
lamination $\Lambda$ as well as the 
corresponding 
subshift $X_\Lambda$ turn out to be both, minimal and uniquely ergodic, so that we have $\sigma^\cal M(\mu) = \lambda\, \mu\,$, where $\lambda > 1$ is the celebrated ``stretching factor'' (= the Perron-Frobenious eigenvalue of $M(\sigma)$) for $h$.

\smallskip
\noindent
(2)
Any word $w \in \cal A^*$ (or rather, its conjugacy class in the free group $F(\cal A)$) 
defines a finite subshift $X_w$ that consists of 
the sequence $w^{\pm \infty} = \ldots w w w \ldots \,\,$ and its finitely many shift translates. 
To $w$ 
there 
is canonically associated a {\em characteristic measure} $\mu_w \in \cal M(\cal A^\Z)$ with support equal to $X_w$ and total measure $\mu_w(X_w) = \mu_w(\cal A^\Z) = |w|$ (see (\ref{eq:charac-m}) for more details).

For any non-erasing morphism $\sigma: \cal A^* \to \cal B^*$ the transferred measure of any characteristic measure is again a characteristic measure, given by
$$\sigma^\cal M(\mu_w) = \mu_{\sigma(w)} \, .$$

\smallskip
\noindent
(3)
For the special case that $\sigma: \cal A^* \to \cal B^*$ extends to a free group automorphism $\phi_\sigma: F(\cal A) \to F(\cal B)$, one can use the well established 1-1 relationship between invariant measures on $\cal A^\Z$ on one hand and currents on $F(\cal A)$ on the other, as well as the similarly well established action of $\Aut(F(\cal A))$ on the current space $\Curr(F(\cal A)) \cong \cal M(\cal A^\Z)$ (see \cite{Ka2} for details). Denoting by $\mu$ also the current on $F(\cal A)$ defined by the invariant measure $\mu$, one has
$$\phi_\sigma(\mu) = \sigma^\cal M(\mu) \, .$$
\end{example}

\bigskip
\noindent
{\em Acknowledgements:}
We would like to thank Fabien Durand and Samuel Petite 
for encouraging remarks and interesting comments,
and the two referees, who have sent us very stimulating questions and useful suggestions.

\section{Notation and conventions}
\label{sec:2}

\subsection{Standard terminology and well known facts}
\label{sec:2.1}

${}^{}$

\smallskip

Throughout this paper we denote by $\cal A, \cal B, \cal C, \ldots$ finite 
non-empty 
sets, called {\em alphabets}. For 
any such alphabet, say $\cal A$, we denote by $\cal A^*$ the free monoid over the set $\cal A$, given by all finite words $w = x_1 x_2 \ldots x_n$ with $x_i \in \cal A$, and equipped with the multiplication defined by concatenation. We denote by $|w| := n$ the {\em length} of any such word. 
The empty word $\epsilon \in \cal A^*$ is defined by $|\epsilon| = 0$; it is the unit element with respect to the multiplication in $\cal A^*$.

The elements $a_i \in \cal A$ are called the {\em letters}; they constitute a set of generators of $\cal A^*$ 
and moreover a basis of the associated free group $F(\cal A)$.
Any 
subset $\cal L \subset \cal A^*$ 
(often assumed to be infinite) 
is called a {\em language} over $\cal A$. 

\smallskip

A {\em monoid morphism} $\sigma: \cal A^* \to \cal B^*$ is well defined by knowing the image 
word $\sigma(a_i) \in \cal B^*$ for each of the letters $a_i \in \cal A$. Conversely, each choice of such image words defines a monoid morphism $\sigma$ as above. The monoid morphism $\sigma$ is {\em non-erasing} if $|\sigma(a_i)|\geq 1$ for each of the letters $a_i \in \cal A$. In this paper we will only consider non-erasing morphisms. Note that any non-erasing morphism $\sigma: \cal A^* \to \cal B^*$ is ``finite-to-one'', i.e. for any element $w \in \cal B^*$ the preimage set $\sigma^{-1}(w)$ is finite.

Every monoid morphism $\sigma: \cal A^* \to \cal B^*$ defines an {\em incidence matrix}
\begin{equation}
\label{incidence-m}
M(\sigma) = (|\sigma(a_j)|_{b_i})_{b_i \in \cal B, a_j \in \cal A}
\end{equation}
where $|\sigma(a_j)|_{b_i}$ denotes the number of occurrences of the letter $b_i \in \cal B$ in the $\sigma$-image of any $a_j \in \cal A$. 
One easily verifies the formula $M(\sigma) = M(\sigma_2) \cdot M(\sigma_1)$ for any composition of monoid morphisms $\sigma = \sigma_2 \circ \sigma_1$.

\smallskip

To any alphabet $\cal A$ there is canonically associated the shift space 
$\cal A^\Z$. Its elements are written as biinfinite words ${\bf x} = \ldots x_{i-1} x_i x_{i+1} \ldots$ with $x_i \in \cal A$. The set $\cal L({\bf x}) \subset \cal A^*$ of all finite subwords 
(called {\em factors}) 
${\bf x}_{[k, \ell]} := x_{k} x_{k+1} \ldots x_\ell$ is the {\em language associated to ${\bf x}$}. The one-sided infinite {\em positive half-word} $x_1 x_2 \ldots $ of $\bf x$ is denoted by ${\bf x}_{[1, \infty)}$.
Similarly, we denote the complementary infinite half-word $\ldots x_{-1} x_0$ of $\bf x$ by ${\bf x}_{(-\infty, 0]}$.

The shift space $\cal A^\Z$ 
is canonically equipped with 
the {\em shift operator} $T: \cal A^\Z \to \cal A^\Z$ which maps the word ${\bf x} = \ldots x_{i-1} x_i x_{i+1} \ldots$ to the word ${\bf y} = \ldots y_{j-1} y_j y_{j+1} \ldots$ given by $y_k = x_{k+1}$ for all indices $k \in \Z$.
Similarly, 
$\cal A^\Z$ is naturally equipped with the product topology 
(with respect to the discrete topology on $\cal A$), which makes 
$\cal A^\Z$ 
into a compact space, indeed a Cantor set
 (unless $\card(\cal A) = 1$). 
For any $w \in \cal A^*$ the {\em cylinder} $[w] \subset \cal A^\Z$ is open and closed (and thus compact); it consists of all biinfinite words ${\bf x} = \ldots x_{i-1} x_i x_{i+1} \ldots$ with ${\bf x}_{[1, |w|]} = w$. The set of all cylinders and their shift-translates constitutes a basis for the topology on $\cal A^\Z$. The shift operator $T$ acts as homeomorphism on $\cal A^\Z$.

\smallskip

A subset $X \subset \cal A^\Z$ is called a {\em subshift} if $X$ is non-empty, closed, and 
if it satisfies $T(X) = X$. 
To any 
infinite 
language $\cal L \subset \cal A^*$ there is canonically associated the subshift $X(\cal L)$ {\em generated by $\cal L$}: It is defined through
$${\bf x} \in X(\cal L) \quad \Longleftrightarrow \quad \cal L({\bf x}) \subset \cal L^f \, ,$$
where $\cal L^f$ denotes the {\em factorial closure} of $\cal L$, which is the language obtained from $\cal L$ by adding in all {\em factors} (= subwords) of any word $w \in \cal L$.
Conversely, every subshift $X \subset \cal A^\Z$ determines an {\em associated subshift language} $\cal L(X) := 
\bigcup\{\cal L({\bf x})\mid {\bf x} \in X\}$, which is infinite and equal to its factorial closure, so that one has $X = X(\cal L(X))$.

A subshift $X$ is {\em minimal} if for any element ${\bf x} \in X$ the set $X$ is equal to 
the closure of the {\em shift-orbit}
$$\cal O({\bf x}) := \{T^n({\bf x}) \mid n \in \Z\}\, .$$
Important particular examples are minimal subshifts that are finite: They consist of a single orbit $\cal O({\bf w})$ that is given by the finitely many shift-translates of a periodic word ${\bf w} 
= \ldots w w w \ldots$. 
To be specific, we fix the indexing of such a periodic word by the requirement that $w$ is a prefix of the positive half-word ${\bf w}_{[1, \infty)}$; in this case ${\bf w}$ is denoted by $w^{\pm\infty}$, 
and we write ${\bf w}_{[1, \infty)} =: w^{+ \infty}$ and ${\bf w}_{(-\infty, 0]}= :w^{- \infty}$.

The space $\Sigma(\cal A)$ of all subshifts $X \subset \cal A^\Z$ is naturally equipped with the partial order given by the inclusion; the minimal elements with respect to this partial order are precisely the minimal subshifts. The shift space $\cal A^\Z$ itself is the only maximal element with respect to this partial order; it is often also called the {\em full shift over $\cal A$}. Similarly, a subshift $X \subset \cal A^\Z$ is sometimes called {\em a subshift over $\cal A$}.

The space $\Sigma(\cal A)$ is also equipped with a natural topology, inherited from the 
canonical embedding $\Sigma(\cal A) \subset \cal P(\cal A^*)$ which is defined by the above bijection between subshifts and their associated languages.
Since the topology of the shift space doesn't play a role in this paper, we 
will not give details here and refer the reader instead to 
\cite{PS}. Just for ``general interest'' we note that the 
subset of $\Sigma(\cal A)$, which consists of all subshifts that are a union of finitely many shift-orbits, is dense in $\Sigma(\cal A)$. More information about the topology on $\Sigma(\cal A)$ can be found in \cite{PS}, where in particular it is shown that $\Sigma(\cal A)$ is a 
Pe\l{}czy\'nski 
space.

\medskip

An {\em invariant measure} $\mu$ on a subshift $X \subset \cal A^\Z$ is a finite Borel measure on $\cal A^\Z$ with support in $X$ that is invariant under $T$, i.e.  
$\mu(T^{-1}(B))=\mu(B)$ for every measurable set $B\subseteq X$. 
Any invariant measure $\mu$ defines a {\em weight function}
$$\omega_\mu: \cal A^* \to \R_{\geq 0}, \,\, w \mapsto \mu([w]) \, ,$$
by which we mean any function $\omega: \cal A^* \to \R_{\geq 0}$ that satisfies the {\em Kirchhoff equalities}
\begin{equation}
\label{Kirchhoff}
\omega(w) = \sum_{a_i \in \cal A} \omega(a_i w) = \sum_{a_i \in \cal A} \omega(w a_i)
\end{equation}
for all $w \in \cal A^*$. Conversely, it is well known (see 
\cite{Barreira}, \cite{Walters})
that every weight function $\omega: \cal A^* \to \R_{\geq 0}$ defines an invariant measure $\mu_\omega$ via $\mu_\omega([w]) := \omega(w)$ for all $w \in \cal A^*$. This gives 
$\mu_{\omega_\mu} = \mu$ and $\omega_{\mu_\omega} = \omega$, and hence a bijective relation between invariant measures and weight functions. This bijection respects the addition and the multiplication with scalars $\lambda \in \R_{\geq 0}\, $, both of which are naturally defined for invariant measures as well as for weight functions. We thus simplify 
our notation by writing 
\begin{equation}
\label{eq-simplify}
\mu(w) := \mu([w]) = \omega_\mu(w) \, .
\end{equation}

\smallskip

An invariant measure $\mu$ is 
{\em ergodic} if $\mu$ can not be written in any non-trivial way as sum $\mu_1 + \mu_2$ of two invariant measures $\mu_1$ and $\mu_2$ (i.e. $\mu_1 \neq 0 \neq \mu_2$ and $\mu_1 \neq \lambda \mu_2$ for any $\lambda \in \R_{> 0}$). An invariant measure is called a {\em probability measure} if $\mu(X) = 1$, which is equivalent to $\underset{a_i \in \cal A}{\sum} \mu([a_i]) = 1$.

We denote by $\cal M(X)$ the set of invariant measures on $X$, and by 
$\cal M_1(X) \subset \cal M(X)$ the subset of probability measures. The set $\cal M(X)$ is naturally equipped with the {\em weak$^*$-topology} (see 
\cite{Barreira}, \cite{Ka2}, \cite{Walters}). 
Using the above bijection to the set of weight functions this topology turns out to be 
equivalent to the topology inherited from the canonical embedding of $\cal M(X)$ into the product space $\R^{\cal A^*}$ given by $\mu \mapsto (\mu(w))_{w \in \cal A^*}$.

It follows that the set $\cal M(X)$ is a convex linear cone which is naturally embedded into the non-negative cone of the infinite dimensional vector space $\R^{\cal A^*}$. The cone $\cal M(X)$  is closed, and  the extremal vectors of $\cal M(X)$ are in 1-1 relation with the ergodic measures on $X$. Furthermore, $\cal M_1(X)$ is compact, and it is the closed convex hull of its extremal points.

It is well known (see 
\cite{Barreira}, \cite{Walters}) 
that for any subshift $X \subset \cal A^\Z$ the set $\cal M(X)$ of invariant measures is not empty. If $\cal M_1(X)$ consists of a single point (which then must be ergodic), then $X$ is called {\em uniquely ergodic}.

\smallskip

An important class of ergodic invariant measures on the full shift $\cal A ^\Z$ 
is given by the {\em characteristic measures} $\mu_w$, for any 
non-empty $w \in \cal A^* \smallsetminus \{\epsilon\}$, 
which are defined as follows: If $w$ is 
{\em not a proper power}, i.e. 
\begin{equation}
\label{eq:prop-pow}
\text{$w \neq w_0^r$ for any $w_0 \in \cal A^*$ and any integer $r \geq 2$,} 
\end{equation}
then $\mu_w$ simply counts for any measurable set $B \subset \cal A^\Z$ the number of intersections of $B$ with the minimal finite subshift $\cal O(w^{\pm \infty})$ associated to $w\,$, 
giving the equality
\begin{equation}
\label{eq:charac-m}
\mu_w(B) \,\, := \,\, \card(B \cap \cal O(w^{\pm\infty})) \, .
\end{equation}
If on the other hand 
$w = {w_0}^r$ for some $w_0 \in \cal A^*$ and some integer $r \geq 2$, where $w_0$ is assumed not to be a proper power, one sets
$$\mu_w \, \, :=\,\,  r \cdot \mu_{w_0} \, .$$
In either case, it follows that $\frac{1}{|w|} \mu_w$ is a probability measure.
The set of weighted characteristic measures $\lambda \, \mu_w$ 
(for any $\lambda > 0$ 
and any 
$w \in \cal A^* \smallsetminus \{\epsilon\}$) is known to be dense in $\cal M(\cal A^\Z)$ (see
\cite{Barreira}, 
\cite{Ka2}, \cite{Walters}).

\medskip

It is well known 
(see \cite{Barreira}, \cite{Walters})
and easy to verify that the support $X_\mu := \supp(\mu)$ of any non-zero 
invariant measure $\mu \in \cal M(\cal A^\Z)$ is a subshift. 
Its 
language $\cal L(X_\mu)$ 
is given by 
\begin{equation}
\label{eq:supp}
w \in \cal L(X_\mu) \quad \Longleftrightarrow \quad \mu(w) > 0
\end{equation}
for any $w \in \cal A^*$.
If $X_\mu$ is uniquely ergodic, then $X_\mu$ is minimal, but the converse is famously wrong (see
\cite{Keane}).
For any characteristic measure $\mu_w$ one has $\supp(\mu_w) = \cal O(w^{\pm\infty})$. 
The support map
\begin{equation}
\label{eq2.2}
\Supp: \cal M(\cal A^\Z)
\smallsetminus \{0\} \to \Sigma(\cal A), \,\, \mu \mapsto X_\mu
\end{equation}
has some nice natural properties: for example, if $\mu_1, \mu_2 \in \cal M(\cal A^\Z)$ and $\lambda_1 > 0\, , \,\, \lambda_2 > 0$ are given, then for $\mu = \lambda_1 \mu_1 + \lambda_2 \mu_2$ one has $X_\mu = X_{\mu_1} \cup X_{\mu_2}$.
Also, every minimal subshift is the support of some measure $\mu$, but (by a variety of reasons) there are non-minimal subshifts that are not in the image of the map $\Supp$. 
An example is given by the subshift that consists of the three orbits $\cal O(a^{\pm \infty}), \cal O(b^{\pm \infty})$ and $\cal O(\ldots aaa bbb \ldots)$.

Unfortunately, however, there is no 
topology on $\Sigma(\cal A)$ which is at the same time Hausdorff and for which the support map is continuous: For any $a_1, a_2 \in \cal A$ and $\mu_n := \frac{1}{n} \mu_{a_1} + \mu_{a_2}$ for any $n \in \N$ we have $\supp(\mu_n) = \cal O(a_1^{\pm \infty}) \cup \cal O(a_2^{\pm\infty})$ but $\lim \mu_n = \mu_{a_2}$ and thus $\supp(\lim \mu_n) = \cal O(a_2^{\pm\infty})$.

An invariant measure $\mu$ on some subshift $X \subset \cal A^\Z$ extends canonically to an invariant measure on all of $\cal A^\Z$. For the corresponding weight functions this extension is obtained by simply declaring $\omega_\mu(w) = 0$ for any $w \notin \cal L(X)$. We will notationally not distinguish between $\mu$ and its canonical extension, or conversely, between $\mu$ and the restriction of $\mu$ to its support. In particular, for any subshift $X \subset \cal A^\Z$ we understand $\cal M(X)$ as canonical subset of $\cal M(\cal A^\Z)$.

\subsection{``Not so standard'' basic facts and terminology}
\label{sec:2.2}

${}^{}$
\smallskip

Let $\sigma: \cal A^* \to \cal B^*$ a non-erasing monoid morphism. Then there is a canonically induced map
\begin{equation}
\label{eq-sigma-Z}
\sigma^\Z: \cal A^\Z \to \cal B^\Z
\end{equation}
that maps a biinfinite word ${\bf x} = \ldots x_{i-1} x_i x_{i+1} \ldots  \in \cal A^\Z$ to the biinfinite word ${\bf y} = \ldots y_{j-1} y_j y_{j+1} \ldots  \in \cal B^\Z$ obtained from concatenating the $\sigma(x_i)$ in the obvious way, starting with the convention 
$$\sigma(x_1) = y_1 \ldots y_{|\sigma(x_1)|}\, .$$
This map $\sigma^\Z$ will in general not inherit any of the properties of $\sigma$ (for instance, 
$\sigma^\Z$ 
is almost never surjective and almost never finite-to-one). However, it satisfies $\sigma^\Z(T({\bf x})) = T^k({\bf y})$ for suitable $k \geq 0$ depending on 
${\bf x}$, 
so that the shift-orbit $\cal O({\bf x})$ has a well defined image shift-orbit $\cal O({\bf y})$. Hence $\sigma^\Z$ induces a map
$$\sigma^T: \cal A^\Z / \langle T \rangle \to \cal B^\Z / \langle T \rangle, \,\, \cal O({\bf x}) \mapsto \cal O(\sigma^\Z({\bf x}))$$
on the associated ``leaf spaces'' (i.e. the set of shift-orbits), which turns out to be a lot more meaningful than the map $\sigma^\Z$ itself.

\begin{rem}
\label{composition}
It follows straight from the definitions that both induced maps $\sigma^\Z$ and $\sigma^T$ behave functorially: in particular, for any two non-erasing monoid morphisms $\sigma_1: \cal A^* \to \cal B^*$ and $\sigma_2: \cal B^* \to \cal C^*$ we have
$$(\sigma_2 \circ \sigma_1)^\Z = \sigma_2^\Z \circ \sigma_1^\Z \quad \text{and} \quad (\sigma_2 \circ \sigma_1)^T = \sigma_2^T \circ \sigma_1^T \, .$$
\end{rem}

\medskip

We now come to 
a third 
map induced by any non-erasing monoid morphism $\sigma: \cal A^* \to \cal B^*$, namely the map
\begin{equation}
\label{most-important}
\sigma^\Sigma: \Sigma(\cal A) \to \Sigma(\cal B) \, , \,\, X \mapsto 
\sigma^\Sigma(X)
\end{equation}
from the space $\Sigma(\cal A)$ of subshifts $X \subset \cal A^\Z$ to the space $\Sigma(\cal B)$ of subshifts  $Y \subset \cal B^\Z$. 
We know of 
three natural ways to define the image $\sigma^\Sigma(X) \subset \cal B^\Z$ under this map, for any given subshift $X \subset \cal A^\Z$, listed below as follows.
Notice
that here the assumption on $\sigma$ to be non-erasing is necessary.

\begin{defn-rem}
\label{image-shift}
If $\sigma$ is non-erasing, 
then the
following three definitions of the {\em image subshift} $Y := \sigma^\Sigma(X)$, for simplicity 
usually 
denoted by $Y = \sigma(X)$, are equivalent:
\begin{enumerate}
\item
$Y$ is the intersection of all subshifts that contain the set $\sigma^\Z(X)$ (which in general is not shift-invariant and hence not a subshift itself).

\item
$Y$ is the closure of the union of all image orbits $\sigma^T(\cal O({\bf x}))$, for any ${\bf x} \in X$. 
In fact 
(see Lemma \ref{closure-not} below), taking the closure in the previous sentence can be omitted.

\item
$Y$ is the subshift generated by the language $\sigma(\cal L(X))$. Thus $Y$ consists of all biinfinite words ${\bf y} \in \cal B^\Z$ with the property that 
every finite factor of $\bf y$ is also a factor of some word in $\sigma(\cal L(X))$.
\end{enumerate}
In particular, the map $\sigma^\Z$ 
restricts/co-restricts to 
a map 
\begin{equation}
\label{eq-shift-map}
\sigma^\Z_X: X \to \sigma^\Z(X) \subset \sigma(X)\, ,
\end{equation}
where 
the inclusion $\sigma^\Z(X) \subset \sigma(X)$ is in general not an equality, due to our convention $\sigma(X) = \sigma^\Sigma(X)$.
\end{defn-rem}

The proof of the equivalence of the statements (1) - (3) above is 
straight forward and hence 
left here to the reader 
(except for the part proved below in Lemma \ref{closure-not}). 
We do however illustrate the terms used above by making them explicit in the following special case:

\begin{example}
\label{Arnaud-wish}
Let $\cal A = \{a, b\}$  and $\cal B = \{c, d\}$, and define $\sigma: \cal A^* \to \cal B^*$ via
$$a \mapsto (cd)^2 = 
cdcd\, ,\,\, b \mapsto (cd)^3 = cdcdcd  \, .$$
We set $X := \{ a^{\pm\infty}, \, b^{\pm \infty}\}$ and obtain:
\begin{itemize}
\item
$\sigma^\Z(X) = \{(cd)^{\pm\infty}\}$
\item
$\sigma^\Sigma(X) = \{(cd)^{\pm\infty}, (dc)^{\pm\infty}\}$
\item
$\sigma^T(\cal O(a^{\pm\infty})) = \sigma^T(\cal O(b^{\pm \infty})) = \cal O((cd)^{\pm \infty}) = \cal O((dc)^{\pm \infty})$ 
\item
$\cal L(X) = \{a^k, b^k \mid k \geq 0\}$
\item
$\sigma(\cal L(X)) = \{(cd)^{k} \mid k \geq 2\}$
\item
$\cal L(\sigma(X)) = \{(cd)^{k}, (dc)^{k}, (cd)^{k}c, (dc)^{k}d \mid k \geq 0\}\, .$
\end{itemize}
\end{example}

The following basic fact 
is used 
below in the proofs of Lemma \ref{5.x.d} 
and of Lemma \ref{8n.1}; 
since we don't know a reference for it, we include here a proof.
Please note that the compactness of $\cal A^\Z$ and the continuity of the map $\sigma^\Z$ yield directly that for any subshift $X$ the set
$\sigma^\Z(X)$ is closed. However, since $\sigma^\Z(X)$ is in general not shift-invariant, this doesn't imply directly that the set $Y$ in the lemma below is closed.

\begin{lem}
\label{closure-not}
(1)
For any 
non-erasing 
monoid morphism $\sigma: \cal A^* \to \cal B^*$ and any subshift $X \subset \cal A^\Z$ the union 
$Y$ of all image orbits $\sigma^T(\cal O({\bf x}))$, for any ${\bf x} \in X$, is a closed subset of $\cal B^\Z$.

\smallskip
\noindent
(2)
In particular, the map $\sigma_X^T: X / \langle T \rangle \to \sigma(X) / \langle T \rangle$ induced by the map $\sigma^T$ is surjective.
\end{lem}

\begin{proof}
We need to show the following fact:  
\begin{enumerate}
\item[(\#)]
For any integer $n \geq 0$ let ${\bf x}(n) \in X$ be a biinfinite word with image ${\bf y}(n) := \sigma^\Z({\bf x}(n)) \in Y$, 
and assume that for suitable shift exponents $k(n) \in \Z$ the sequence of the biinfinite words $T^{k(n)}({\bf y}(n))$ converges to some ${\bf y} \in Y$. Then there exists a biinfinite word ${\bf x} \in X$ and a subsequence $({\bf x}(n_m))_{m \in \N}$ such that for a suitable set of integer exponents $\ell_m \in \Z$ one has
$${\bf x} = \lim_{m \to \infty}  T^{\ell_m}({\bf x}(n_m))$$
as well as $\sigma^\Z({\bf x}) = T^k({\bf y})$ for some $k \in \Z$.
\end{enumerate}
In order to prove (\#) we first observe that without loss of generality we can replace any ${\bf x}(n)$ by some shift-translate and thus achieve that the set of exponents $k(n)$ is bounded, and indeed, that $0 \leq k(n) < 
|\sigma(x_0)(n)|$, 
where $x_0(n) \in \cal A$ is the 0-th letter of the biinfinite word ${\bf x}(n) = \ldots x_{-1}(n) x_0(n) x_1(n) \ldots $ for any $n \in \N$.

Hence, 
by extracting a subsequence of the ${\bf x}(n)$, we can 
achieve that the letter $x_0 = x_0(n)$ as well as the exponent $k = k(n)$ is independent of $n$. We now use the assumption that the $T^k({\bf y}(n))$ converge, in order to extract 
again a 
subsequence of the previous subsequence of the ${\bf x}(n)$ in order to ensure that the letters $x_{-1}(n)$ and $x_1(n)$ adjacent to $x_0(n)$ on ${\bf x}(n)$ are also independent of $n$. We 
iteratively 
proceed in this manner and extract finally a diagonal subsequence that defines a biinfinite word ${\bf x}$ which is the limit of some subsequence $({\bf x}(n_m))_{m \in \N}$ of the ${\bf x}(n)$. From our construction and the definition of the map $\sigma^\Z$ we see directly that 
$\sigma^\Z({\bf x}) = T^{-k}({\bf y})$.
\end{proof}

\begin{rem}
\label{4.3}
The concept of an image subshift given in Definition-Remark \ref{image-shift} is a very natural one; it is used frequently, in particular in the $S$-adic context. A systematic treatment, however, doesn't seem to be available anywhere.
One derives easily from the above definitions the following properties of the map (see (\ref{most-important})) 
$$\sigma^\Sigma: \Sigma(\cal A) \to \Sigma(\cal B)$$
on the subshift spaces over $\cal A$ and $\cal B$ respectively. We recall that for simplicity we allow ourselves to denote the image of a subshift $X \subset \cal A^\Z$ by $\sigma(X)$ rather than by $\sigma^\Sigma(X)$.
\begin{enumerate}
\item
The map $\sigma^\Sigma$ is functorial. In particular, for any second morphism $\sigma': \cal B^* \to \cal C^*$ one has
$$\sigma'(\sigma(X)) = \sigma' \circ \sigma(X) \, .$$
\item
The map $\sigma^\Sigma$ respects the inclusion: For any two subshifts $X' \subset X$ in $\Sigma(\cal A)$ one has $\sigma(X') \subset \sigma(X)$.
For any third subshift $X''$ in $\Sigma(\cal A)$ one has $\sigma(X \cap X'') \subset \sigma(X) \cap \sigma(X'')$ and $\sigma(X \cup X'') = \sigma(X) \cup \sigma(X'')$; the analogous statements hold for infinite intersections and for the closure of infinite unions.

\item
The map $\sigma^\Sigma$ respects the subshift closure: If $\cal L \subset \cal A^*$ is an infinite set, then the subshift $X(\cal L) \in \Sigma(\cal A)$ generated by $\cal L$ has as image the subshift generated by $\sigma(\cal L)$: 
$$\sigma(X(\cal L)) = X(\sigma(\cal L))$$
\item
For any subshift $X \in \Sigma(\cal A)$ the map $\sigma^\Sigma$ induces on the subset $\Sigma(X) \subset \Sigma(\cal A)$ of subshifts 
contained in $X$ a map
$$\sigma^\Sigma_X: \Sigma(X) \to \Sigma(\sigma(X)) \, .$$
\item
For any subshift $Y \subset \cal B^\Z$ the preimage $(\sigma^\Z)^{-1}(Y) \subset \cal A^\Z$ is 
either the empty set, or else it is a subshift over $\cal A$.
Alternatively this subshift is obtained as union of all subshifts $X_i$ with $\sigma^\Sigma(X_i) \subset Y$. For simplicity we denote this subshift by $\sigma^{-1}(Y) \subset \cal A^\Z$.
\item
By 
considering for any subshifts $X \subset \cal A^\Z$ and $Y \subset \sigma(X)$ the subshift $X \cap \sigma^{-1}(Y)$ we observe that the above map $\sigma^\Sigma_X: \Sigma(X) \to \Sigma(\sigma(X))$ is surjective. It also preserves the partial order given by the inclusion of subshifts.
\item
In particular, if $X$ is minimal, then so is $\sigma(X)$.
\item
The 
map $\sigma^\Sigma$ is continuous with respect to the canonical topology on the 
subshift spaces.
\end{enumerate}
\end{rem}

\begin{rem}
\label{2.5d}
Although not central to the topics evoked in this paper, for completeness we would like to state here how a non-erasing morphism $\sigma: \cal A^* \to \cal B^*$ acts on the complexity and the topological entropy of a subshift. Recall that for any subshift $X \subset \cal A^\Z$ the {\em 
complexity} (also called {\em word complexity}) is given by the function $p_X: \N \to \N$, 
defined via 
$p_X(n) = \card\{w \in \cal L(X) \mid |w| = n\}$. The {\em topological entropy} of $X$ is defined as $h_X = \underset{n \to \infty}{\lim}\frac{\log\, p_X(n)}{n}$. A fairly standard exercise 
(see for instance \cite{Lu22}, Lemma 2.1, or \cite{HL}, 
Lemma 2.9 and Proposition 2.11)
shows, for $Y := \sigma(X)\, $:
\begin{enumerate}
\item
Setting 
$
||\sigma || 
:= \max\{|\sigma(a_k) \mid a_k \in \cal A\}$ one has
$$p_Y(n) \,\, \leq \,\,  
||\sigma||\cdot p_X(n) \, .$$
\item
As direct consequence one derives
$$h_Y \,\, \leq \,\, h_X \, .$$
\end{enumerate}
\end{rem}

\smallskip

\subsection{About injectivity}
\label{sec:2.3}

${}^{}$

Injectivity of monoid morphisms $\sigma$ and of their induced maps $\sigma^\Z, \sigma^T$, and $\sigma^\Sigma$ is an important and often tricky issue. 
We start out in sub-Subsection \ref{sec:2.3.1} to list several problems and give examples where these often undesired phenomena do occur. In sub-Subsection \ref{sec:2.3.2}, which can be read independently, we will then define the subtle notion of ``shift-period preservation'', 
which will be used in Sections \ref{sec:5d} and \ref{sec:6d} below as well as in \cite{BHL2.8-II}.

\subsubsection{Typical injectivity problems}
\label{sec:2.3.1}

${}^{}$

\smallskip

For starters, we have the following well known phenomenon:

\begin{rem}
\label{2.5d}
There exist non-erasing morphisms $\sigma: \cal A^* \to \cal B^*$ which are injective, while the associated free group homomorphism $F(\cal A) \to F(\cal B)$ is not injective. An example 
is given by
$$\sigma: \{a, b, c\}^* \to \{d, e\}^*\, ,\,\, a \mapsto d d e, \,\, b \mapsto d e e, \,\, c \mapsto d e d e \,$$
which maps $c (a c^{-1} b)^{-2}$ to the neutral element $1 \in F(\{d, e\})$.
\end{rem}

A second, also well known disturbance comes from the fact that injective free group morphisms need not induce injective maps on the set of conjugacy classes. This happens for example 
for the morphism induced by the quotient map, 
if two boundary curves of a surface (with at least one more boundary component to ensure free fundamental groups) are glued together. In a classical monoid morphism situation we observe that for the ``Thue-Morse 
morphism''
\begin{equation}
\label{eq2.7d}
\sigma_\text{TM}: \{a, b\}^* \to \{c, d\}^* \, , \,\,\, a \mapsto cd, \,\, b \mapsto dc
\end{equation}
the images $\sigma_\text{TM}(a)$ and $\sigma_\text{TM}(b)$ are conjugate in $F(\{c, d\})
$, 
while $\sigma_\text{TM}$ is injective on all of $F(\{a, b\})$. 
This yields:

\begin{rem}
\label{2.6d}
There exist non-erasing monoid morphisms $\sigma$ which are injective, but for which the induced map $\sigma^T$ is not injective. For example, for the Thue-Morse morphism one has $\cal O(\sigma_\text{TM}(a^{\pm \infty}))= \cal O(\sigma_\text{TM}(b^{\pm \infty}))$.
\end{rem}

Two further, more subtle ``injectivity disturbances'' can be observed by the following two examples, 
despite the fact that in both cases the induced map on the shift-orbits is injective. Indeed, in both cases we consider a subshift which consists of a single periodic shift-orbit: in the first case the map $\sigma^\Z$ restricted to this orbit is not injective, while in the second case it is, but a minimal period in the preimage orbit is not mapped to a minimal period in the image.

\begin{rem}
\label{2.7d}
(1)
The morphism $\sigma_1: \{a, b\}^* \to \{c\}^*\, , \,\, \sigma_1(a) = \sigma_1(b) = c$ maps the orbit $\cal O((ab)^{\pm \infty})$ to $\cal O( c^{\pm \infty} )$, but the restriction of $\sigma_1^\Z$ to $\cal O( (ab)^{\pm \infty} )$ is not injective.

\smallskip
\noindent
(2)
The morphism $\sigma_2: \{a\}^* \to \{b\}^*\, , \,\, \sigma_2(a) = b^2$ maps the orbit $\cal O( a^{\pm \infty} )$ 
injectively 
to $\cal O( b^{\pm \infty} )$, but the ``shift-period'' is not preserved: 
For both, $a^{\pm \infty}$ and $b^{\pm \infty}$ the minimal shift-period has length 1, but $\sigma_2$ maps the shift-period $a$ of $a^{\pm \infty}$ to a word of length 2, which is hence not the minimal shift-period of $b^{\pm \infty}$.
\end{rem}

\subsubsection{Shift-period preservation}
\label{sec:2.3.2}

${}^{}$

\smallskip

For any periodic word ${\bf x} \in \cal A^\Z$ we define the {\em shift-period exponent} of ${\bf x}$ 
to be the smallest integer $k \geq 1$ such that $T^k({\bf x}) = {\bf x}$. If ${\bf x} = w^{\pm \infty}$ for some $w \in \cal A^*$, then $k$ divides $|w|$. If $w$ can not be written as proper power (see (\ref{eq:prop-pow})), then the shift-period exponent of $\bf x$ is given by $k = |w|$. In this case any cyclic permutation of $w$ will be called a {\em shift-period} of the periodic word ${\bf x}$.

\begin{defn}
\label{period-preserving}
A morphism $\sigma: \cal A^* \to \cal B^*$ is said to {\em preserve the shift-period} 
of some biinfinite periodic word ${\bf x} \in\cal A^\Z$ if any shift-period $w \in \cal A^*$ of ${\bf x}$ is mapped by $\sigma$ to a shift-period of the image word $\sigma^\Z({\bf x}) \in \cal B^\Z$.

In other words, if ${\bf x} = w^{\pm \infty}$ and $\sigma(w)$ is a proper power, then so is $w$
(or else $\sigma$ doesn't preserve the shift-period of ${\bf x}$).

\end{defn}

\begin{rem}
\label{2.9d}
In the special case 
where $|\sigma(a_i)| = 1$ for all $a_i  \in \cal A$ we observe that $\sigma$ preserves the shift-period of some periodic biinfinite word ${\bf x}$ if and only if the shift-period exponents of $\bf x$ and of its image $\sigma^\Z({\bf x})$ agree.
\end{rem}

\section{The measure transfer}
\label{sec:image-measure}

In this section we will carefully define for any non-erasing monoid morphism $\sigma: \cal A^* \to \cal B^*$ and any invariant measure $\mu$ on $\cal A^\Z$ a shift-invariant ``image measure'' 
$\mu^\sigma$ on $\cal B^\Z$. The simplest and most natural way to understand this measure transfer is achieved by decomposing the given morphism $\sigma$ in a canonical way into two morphisms of very elementary type. We start our detailed presentation by considering first each of these two elementary morphism types separately.

\subsection{Subdivision morphisms}
\label{sec:subdivision-morphs}

${}{}$
\smallskip

Let $\cal A$ be a finite alphabet, and let $\ell: \cal A \to \Z_{\geq 1}$ be any map, called {\em subdivision length function}. We now define a new {\em subdivision alphabet} $\cal A_{\ell}$ which consists of letters $a_i(k)$ for any $a_i \in \cal A$ and any $k \in \{1, \ldots , \ell(a_i)\}$. We then define the associated {\em subdivision morphism} 
given by
\begin{equation}
\label{eq3.1}
\pi_{\ell}: \cal A^* \to \cal A_{\ell}^*\, , \,\, a_i \mapsto a_i(1) \ldots a_i({\ell(a_i))} \, .
\end{equation}

\begin{rem}
\label{3.1}
(1)
The name and the intuition here comes from picturing $\cal A$ as edge labels of an oriented ``rose'' 
$R_\cal A$, i.e. a 1-vertex graph with $\card (A)$ 
oriented edges. Then any edge with label $a_i$ is subdivided by introducing $\ell(a_i) - 1$ new vertices in its interior, and by labeling the obtained new edges (in the order given by the orientation) by $a_i(1), a_i(2), \ldots, a_i(\ell(a_i))$. Then any edge path $\gamma(w)$ in $R_\cal A$, which reads off a word $w \in \cal A^*$, will after the subdivision read off the word $\pi_{\ell}(w)$.

\smallskip
\noindent
(2)
The natural ``geometrization'' of the {\em subdivision monoid} $\cal A_{\ell}^*$ as rose $R_{\cal A_{\ell}}$ is however not quite the above subdivision of the rose $R_\cal A$, but is obtained from the latter by identifying all subdivision vertices into a single vertex. This is reflected by the fact that the subdivision morphism $\pi_{\ell}$ defined above is 
not surjective (except in the trivial case where all $\ell(a_i) = 1$ so that $\pi_\ell$ is a bijection).
Its image generates a subshift of finite type in $\cal A_{\ell}^\Z$.
\end{rem}

From the above definitions we deduce directly the following:

\begin{lem}
\label{3.2}
Let $\pi_{\ell}: \cal A^* \to \cal A_{\ell}^*$ be a subdivision morphism as in (\ref{eq3.1}). Then the following holds:
\begin{enumerate}
\item
The monoid morphism $\pi_{\ell}$ is injective.
\item
The induced map $\pi_{\ell}^\Z: \cal A^\Z \to \cal A_{\ell}^\Z$ is injective.
\item
The map $\pi_{\ell}^T$ induced on shift-orbits is injective.
\item
The morphism $\pi_\ell$ preserves the shift-period of any biinfinite periodic word ${\bf x} \in \cal A^\Z$.
\item
The map $\pi_{\ell}^\Sigma$ induced on subshifts over $\cal A$ is injective.
\end{enumerate}
\end{lem}

\begin{proof}
All the listed properties follow directly from the definition of the map $\pi_\ell$, since for any element $w' \in \pi_\ell(\cal A^*)$ the (uniquely determined) preimage $w \in \cal A^*$ is directly visible through replacing every factor $a_i(1) \ldots a_i(\ell (a_i))$ by the letter $a_i\, $.
\end{proof}

\begin{defn}
\label{3.3}
Let $\pi_{\ell}: \cal A^* \to \cal A_{\ell}^*$ be a subdivision morphism as in (\ref{eq3.1}), and let $\mu$ be an invariant measure on $\cal A^\Z$. Consider the weight function $\cal A^* \to \R_{\geq 0}\, ,\,\,  w \mapsto \mu([w])$ associated to $\mu$, which for simplicity 
is also denoted by $\mu$ (see (\ref{eq-simplify})).

Define a function $\mu_{\ell}: \cal A_{\ell}^* \to \R_{\geq 0}$ by
\begin{equation}
\label{eq3.late}
\mu_{\ell}(w) = \mu(\hat w)\, ,
\end{equation}
where $\hat w \in \cal A^*$ is the shortest word in $\cal A^*$ such that $\pi_{\ell}(\hat w)$ contains $w$ as factor. If such $\hat w$ exists, then it is uniquely defined by $w$.
If there is no such word $\hat w$, we set formally $\mu(\hat w) = 0$ and thus 
$\mu_\ell(w) = 0$.

It is shown in Lemma \ref{old-rem} just below that 
the function $\mu_{\ell}$ is a weight function, so that (see Section \ref{sec:2.1}) it defines an invariant measure on $\cal A_{\ell}^\Z$ which will also be denoted by $\mu_{\ell}$. We call $\mu_\ell$ the {\em subdivision measure} defined by $\mu$.
\end{defn}

\begin{lem}
\label{old-rem}
The function $\mu_{\ell}$ inherits from $\mu$ the Kirchhoff equalities (\ref{Kirchhoff}), so that $\mu_{\ell}$ is itself a weight function. 
\end{lem}

\begin{proof}
By symmetry it suffices to prove the first equality of (\ref{Kirchhoff}) for the function $\mu_\ell$. We have to distinguish three cases:

If $w \in \cal A_\ell^*$ is not a factor of any element from $\pi_\ell(\cal A^*)$, then the same is true for $x w$ for any $x \in \cal A_\ell$. In this case both, the left and the right hand side of the desired equality, are equal to 0.

If $w$ is a factor of some element in $\pi_\ell(\cal A^*)$, then we consider the first letter $a_i(k)$ of $w$. If we have 
$k \geq 2$, 
then there is only one letter $x \in \cal A_\ell$ such that $x w$ is a factor of some element from $\pi_\ell(\cal A^*)$, namely $x = a_i(k-1)$. In this case we have $\hat{x w} = \hat w$, so that again the desired equality holds.

Finally, if the first letter of $w$ is equal to some $a_i(1)$, then $x w$ is a factor of some element from $\pi_\ell(\cal A^*)$ precisely if $x = a_j(\ell(j))$ for any of the $a_j \in \cal A$. In this case we have $\hat{x w} = a_j \hat w$, so that the first Kirchhoff equality for $\mu(\hat w)$ gives directly the desired equality for $\mu_\ell(w)$.
\end{proof}

\begin{rem}
\label{3.4}
The subdivision measure defined by a probability measure will in general not be probability: Unless $\ell$ is the constant function with value 1, for the total measure we will have
$$\mu_{\ell}(\cal A_{\ell}^\Z) > 1\, .$$
In fact, one easily derives from 
(\ref{eq3.late}) that
$\mu_{\ell}(a_i(k)) =\mu(a_i)$, 
which yields 
$$\mu_{\ell}(\cal A_{\ell}^\Z) = \sum_{a_i(k) \,\in\, \cal A_\ell} \mu_\ell(a_i(k))
= \sum_{a_i \in \cal A} \ell(a_i) \cdot \mu(a_i) \, .$$
\end{rem}

\smallskip

\subsection{Letter-to-letter morphisms}
\label{sec:letter-to-letter}

${}{}$
\smallskip

Recall that a monoid morphism $\alpha: \cal A^* \to \cal B^*$ is called {\em letter-to-letter} if for any letter $a_i \in \cal A$ the length of its image is equal to $|\alpha(a_i)| = 1$. In other words: $\alpha$ is induced by a map $\alpha_\cal A: \cal A \to \cal B$ on the alphabets\,\footnote{\, Some authors (see for instance 
\cite{Durand-Perrin}) require in addition that a letter-to-letter morphism must be surjective. All letter-to-letter morphisms occurring in this paper are indeed surjective, but formally we do not need this condition anywhere.}.

It follows directly that both, the letter-to-letter morphism $\alpha: \cal A^* \to \cal B^*$ and the induced map $\alpha^\Z: \cal A^\Z \to \cal B^\Z$, are injective and/or surjective if and only if $\alpha_\cal A$ is injective and/or surjective. Furthermore, $\alpha^\Z$ commutes with the shift maps (both denoted by $T$) on $\cal A^\Z$ and $\cal B^\Z\,$,
thus giving
$$T \circ  \alpha = \alpha \circ T \, .$$
As a consequence, the image $\alpha^\Z(X)$ of any subshift $X \subset \cal A^\Z$ is equal to the image subshift $\alpha^\Sigma(X)$ over $\cal B$. Furthermore, for any invariant measure $\mu$ on $\cal A^\Z$ the classical push-forward measure $\alpha_*(\mu)$ is an invariant measure on $\cal B^\Z$, with support equal to $\alpha^\Z(\supp(\mu))$. In particular, if $\mu$ is a probability measure, then so is $\alpha_*(\mu)$.

According to the defining equation 
for the push-forward measure, 
$\mu_*(S) := \mu(f^{-1}(S))$ for any measurable set $S$ in the range of 
any measurable map $f$, 
we obtain for the weight function associated to $\alpha_*(\mu)$ 
and for any $w \in \cal B^*$
the finite sum decomposition
\begin{equation}
\label{eq3.2}
\alpha_*(\mu)(w) = \sum_{u \,\in \,\alpha^{-1}(w)} \mu(u) \, .
\end{equation}
Note 
here that any $u \in \alpha^{-1}(w)$ has length $|u| = |w|$.

\subsection{The induced measure morphisms}
\label{sec:measure-morph}

${}{}$
\smallskip

We now consider an arbitrary non-erasing monoid morphism $\sigma: \cal A^* \to \cal B^*$ (as usual for finite alphabets $\cal A$ and $\cal B$). Then $\sigma$ defines a subdivision length function $\ell_\sigma: \cal A \to \Z_{\geq 1}\,$, given by 
$$\ell_\sigma(a_i) := |\sigma(a_i)|$$
for 
any $a_i \in \cal A$.
This gives 
a subdivision alphabet $\cal A_\sigma := \cal A_{\ell_\sigma}$ as well as a subdivision morphism
$$\pi_\sigma := \pi_{\ell_\sigma} : \cal A^* \to \cal A_{\sigma}^* \, .$$
Furthermore, $\sigma$ defines a letter-to-letter morphism given by
$$\alpha_\sigma: \cal A_{\sigma}^* \to \cal B^* \, , \,\, a_i(k) \mapsto \sigma(a_i)_k \, ,$$
for any $a_i \in \cal A$ and $1 \leq k \leq |\sigma(a_i)|$, where 
$\sigma(a_i)_k \in \cal B$ 
denotes the $k$-th letter of the word $\sigma(a_i) \in \cal B^*$.

\begin{defn-rem}
\label{3.5}
(1) From the above definitions we observe that any non-erasing monoid morphism $\sigma:  \cal A^* \to \cal B^*$ admits a {\em canonical decomposition} 
\begin{equation}
\label{eq:can-dec}
\sigma = \alpha_\sigma \circ \pi_\sigma
\end{equation}
as product of a subdivision morphism $\pi_\sigma$ with a subsequent letter-to-letter morphism $\alpha_\sigma$.

\smallskip
\noindent
(2)
As a consequence, the morphism $\sigma$ induces a canonical {\em measure transfer} $\sigma^\cal M: \cal M(\cal A^\Z) \to \cal M(\cal B^\Z)$ which maps any invariant measure $\mu$ on $\cal A^\Z$ to the measure $(\alpha_\sigma)_*(\mu_{\ell_\sigma})$, 
which for simplicity will sometimes be denoted by $\mu^\sigma$.
For the associated weight function $\sigma^\cal M(\mu)$ on $\cal B^*$ we obtain, for any $w \in \cal B^*$, the formula
\begin{equation}
\label{image-measure}
\sigma^\cal M(\mu)(w)  \,\,\,
[ =:\mu^\sigma(w)] = 
\sum_{u \,\in\, {\alpha_\sigma^{-1}(w)}} \mu_{\ell_\sigma}(u)\, =
\sum_{u \,\in\, {\alpha_\sigma^{-1}(w)}} \mu(\hat u)\, ,
\end{equation}
where $\hat u \in \cal A^*$ is defined above in Definition \ref{3.3}. Recall from Definition \ref{3.3} that if $\hat u$ doesn't exist for some $u \in {\alpha_\sigma^{-1}(w)}$, then one has formally set $\mu(\hat u) = 0$.
\end{defn-rem}

The above canonical decomposition of an arbitrary morphism 
has been used previously; it can be found for instance in Lemma 3.4 of \cite{DDMP}.

\subsection{Basic properties of the measure transfer map}
\label{sec:3.4}

${}^{}$

\smallskip

In this subsection we want to show some first properties of the measure transfer map $\sigma^\cal M$ defined in the previous subsection. We start out with some basic facts: their proof is an elementary (and not very illuminating) exercise, based on the canonical decomposition (\ref{eq:can-dec}) of any morphism $\sigma$; it is hence not carried through here.

\begin{lem}
\label{3.6}
Let $\sigma: \cal A^* \to \cal B^*$ be a non-erasing monoid morphism. Then the induced measure transfer map $\sigma^\cal M$ has the following properties:
\begin{enumerate}
\item[(a)]
The map $\sigma^\cal M: \cal M(\cal A^\Z) \to \cal M(\cal B^\Z)$ is 
$\R_{\geq 0}$-linear.
\item[(b)]
The map $\sigma^\cal M$ is functorial: For any second non-erasing monoid morphism $\sigma': \cal B^* \to \cal C^*$ one has $(\sigma' \sigma)^\cal M = {\sigma'}^\cal M {\sigma}^\cal M$.
\item[(c)]
The image $\sigma^\cal M(\mu)$ of a probability measure $\mu$ on $\cal A^\Z$ is in general not a probability measure on $\cal B^\Z$. 
\item[(d)]
For any word $w \in \cal A^*$ the characteristic measure $\mu_w$ is mapped by $\sigma^\cal M$ to the characteristic measure $\mu_{\sigma(w)}$.
\item[(e)]
For any word $w \in \cal A^*$ the measures of corresponding cylinders $[w] \subset \cal A^\Z$ and $[\sigma(w)] \subset \cal B^\Z$ satisfy the 
inequality
\begin{equation}
\label{eq3.cylinders}
\mu([w]) \leq \sigma^\cal M(\mu)([\sigma(w)]) \, .
\end{equation}
If $\sigma$ is a subdivision morphism, then (\ref{eq3.cylinders}) becomes an equality, but for a letter-to-letter morphism the inequality (\ref{eq3.cylinders}) will in general be strict.
\qed
\end{enumerate}
\end{lem}

For the next observation we recall from Section \ref{sec:2.1} that the classical weak$^*$-topology on the space $\cal M(\cal A^\Z)$ is equivalent to the topology induced by the product topology on the space $\R_{\geq 0}^{\cal A^*}$, via the embedding $\cal M(\cal A^\Z) \subset \R_{\geq 0}^{\cal A^*}$ given by $\mu \mapsto (\mu([w]))_{w \in \cal A^*}$.

\begin{lem}
\label{3.6.5d}
For any non-erasing monoid morphism $\sigma: \cal A^* \to \cal B^*$ the induced map $\sigma^\cal M: \cal M(\cal A^\Z) \to \cal M(\cal B^\Z)$ is continuous.
\end{lem}

\begin{proof}
Following (\ref{eq:can-dec}) we decompose $\sigma$ as product $\sigma = \alpha_\sigma \circ \pi_\sigma$. In order to show that the first factor of this decomposition, the morphism $\pi_\sigma: \cal A^* \to \cal A_\sigma^*$, induces a continuous map on $\cal M(\cal A^\Z)$, we recall that by definition this map is defined via $\mu \mapsto \mu_{\ell_\sigma}$, with $\mu_{\ell_\sigma}(w) = \mu(\hat w)$ for any $w \in \cal A_\sigma^*$, where $\hat w$ is the shortest word in $\cal A^*$ such that $w$ is a factor of 
$\pi_\sigma(\hat w)$. It follows directly that small 
variations 
of $\mu$ imply small 
variations 
of $\mu_{\ell_\sigma}$, so that 
$\pi_\sigma$ induces a continuous map $\cal M(\cal A^\Z) \to \cal M(\cal A_\sigma^\Z)$.

The second factor $\alpha_\sigma: \cal A_\sigma^* \to \cal B^*$ of the above decomposition induces a map $\cal M(\cal A_\sigma^\Z)  \to \cal M(\cal B^\Z)$ that is given by the classical push-forward definition $\mu' \mapsto (\alpha_\sigma)_*(\mu')$ for any $\mu' \in \cal M(\cal A_\sigma^\Z)$, for which the continuity is well known (and easy to prove, by a similar approach as used in the previous paragraph).

It follows that the composition $\cal M(\cal A^\Z) \to \cal M(\cal A_\sigma^\Z) \to \cal M(\cal B^\Z)\, , \,\, \mu \mapsto \mu_{\ell_\sigma} \mapsto (\alpha_\sigma)_*(\mu_{\ell_\sigma})$ is continuous. But this is precisely how the map $\sigma^\cal M: \cal M(\cal A^\Z) \to \cal M(\cal B^\Z)$ is defined (see Definition-Remark \ref{3.5} (2)).
\end{proof}

\begin{rem}
\label{3.8.5d}
(1)
Using the density of the set of weighted characteristic measures $\lambda \mu_w$  (see Section \ref{sec:2.1}) within $\cal M(\cal A^\Z)$ on one hand and the continuity of the map $\sigma^\cal M$ on the other, one can use  
statement (d) of Lemma \ref{3.6} in order to 
alternatively 
determine the image $\sigma^\cal M(\mu)$ for any invariant measure $\mu$ on $\cal A^\Z$ 
as 
limit of weighted characteristic measures. 
This methods has been used for instance in \cite{Ka2} to define the map induced by an automorphism of a free group $\FN$ on the space of currents on $\FN$.
The approach presented here, however, has many practical advantages; for instance it is more efficient 
for most computations.

\smallskip
\noindent
(2)
A third alternative to determine the image $\sigma^\cal M(\mu)$ for any invariant measure $\mu$ on $\cal A^\Z$ 
is given in \cite{BHL2.8-II} by means of 
everywhere growing $S$-adic developments
and vector towers 
over them.
\end{rem}

We also observe:

\begin{lem}
\label{3.7d}
Let $\sigma: \cal A^* \to \cal B^*$ be a non-erasing monoid morphism, and let 
$\mu \in \cal M(\cal A^\Z)$ be an invariant non-zero measure 
with support contained in some subshift $X \subset \cal A^\Z$ that has image subshift $Y: = \sigma(X) \subset \cal B^\Z$. Then the 
transferred 
measure $\sigma^\cal M(\mu) \,\,[= \mu^\sigma]$ has support in $Y$. More precisely, using the terminology from (\ref{eq2.2}), we have
\begin{equation}
\label{eq3.5}
\supp(\sigma^\cal M(\mu)) = \sigma^\Sigma(\supp(\mu)) \, .
\end{equation}
In particular, the morphism $\sigma$ induces for any subshift $X \subset \cal A^\Z$ a well defined map
$$\sigma^\cal M_X: \cal M(X) \to \cal M(\sigma(X))$$
which satisfies all the properties analogous to statements (a) - (e) of Lemma \ref{3.6}, as well as to Lemma \ref{3.6.5d}.
\end{lem}

\begin{proof}
Recall first from 
equivalence (\ref{eq:supp}) 
that the language of the support of any shift-invariant measure is given by all words with positive measure of the associated cylinder. 

In particular, any $w \in \cal A^*$ belongs to $\cal L(\supp(\mu))$ if and only if $\mu(w) > 0$. In this case $\sigma(w)$ belongs to $\cal L(\sigma^\Sigma(\supp(\mu)))$, and any $w' \in \cal L(\sigma^\Sigma(\supp(\mu)))$ is a factor of some such $\sigma(w)$. 
From (\ref{eq3.cylinders}) we know $\mu^\sigma(\sigma(w)) \geq \mu(w) $, which implies that $\sigma(w)$ belongs to $\cal L(\supp(\mu^\sigma)))$.
Since $\mu^\sigma(w') \geq \mu^\sigma(\sigma(w))$, the same applies to any factor $w'$ of $\sigma(w)$.
We thus obtain
$$\cal L(\sigma^\Sigma(\supp(\mu))) \subset \cal L(\supp(\mu^\sigma)) \, .$$

Conversely 
(again using (\ref{eq:supp})), 
any $w' \in \cal B^*$ belongs to $\cal L(\supp(\mu^\sigma))$ if and only if $\mu^\sigma(w') > 0$. In this case there exists a word 
$w 
\in \cal A_\sigma^*$ with $\alpha_\sigma(w
) = w'$, and with $\mu_{\ell_\sigma}(w
) > 0$. Hence 
(see formula (\ref{image-measure})) 
there is a word $\hat w \in \cal A^*$ with $\mu(\hat w) > 0$ such that $\pi_\sigma(\hat w)$ contains $w
$ as factor. It follows that $w'$ is a factor of $\sigma(\hat w)$ and that $\hat w \in \cal L(\supp(\mu))$. Since $\sigma(\cal L(\supp(\mu))) \subset \cal L(\sigma^\Sigma(\supp(\mu)))$, this implies
$$\cal L(\supp(\mu^\sigma)) \subset \cal L(\sigma^\Sigma(\supp(\mu))) \, .$$

Hence we have $\cal L(\supp(\mu^\sigma)) = \cal L(\sigma^\Sigma(\supp(\mu)))$, which implies the equality (\ref{eq3.5}).
\end{proof}

The following discussion, suggested to us by a comment of a referee, concerns a potentially alternative approach to the measure transfer map. 
We use here the terminology from \S 3 of \cite{Durand-Perrin}.

\begin{rem}
\label{suggestion-from-referee}
Given 
a subshift $X \subset \cal A^\Z$, one may consider the group 
$H(X, \Z) = C(X, \Z)/
\partial_T C(X, \Z)$, where $C(X, \Z)$ denotes the abelian group of continuous integer valued functions on $X$, and 
$\partial_T(X)$ is the subgroup generated by functions of the form $f \circ T - f$. 
To every invariant measure $\mu$ on $X$ one associates the linear form 
$$\alpha_\mu: H(X, \Z) \to \R\, , \,\, f \mapsto \int f d\mu \, ,$$
i.e. an element in the dual group $H(X, \R)^* := {\rm Hom}(H(X, \Z), \R)$.

The group $H(X, \Z)$ inherits from $C(X, \Z)$ the canonical order as well as the ``order unit'' given by the characteristic function ${\bf 1}_X$ (see \S 2.1.1 of \cite{Durand-Perrin}). 
Under the mild additional assumption that $X$ is recurrent (i.e. there exists some ${\bf x} \in X$ with positive half-word ${\bf x}_{[1, +\infty)}$ that is dense in $X$),  the group $H(X, \Z)$ is a ``unital ordered group'', sometimes denoted by $K^0(X, T)$ (see \cite{Durand-Perrin}, \S 3.3). An order- and unit-preserving morphism from $K^0(X, T)$ to $\R$ is called a {\em state}. The above map $\mu \mapsto \alpha_\mu$ gives a well known bijection (see Theorem 3.9.3 of \cite{Durand-Perrin}) from the set of invariant probability measures on $X$ to the set of states on $K^0(X, T)$. This bijection extends to an $\R_{\geq 0}$-linear isomorphism $i_X^\cal M$ from $\cal M(X)$ to the non-negative cone $H_+^*(X, \R)$ in the dual group $H^*(X,\R)$.

Given now a morphism $\sigma: \cal A^* \to \cal B^*$, then for the image subshift $Y = \sigma(X)$ the measure transfer map defines canonically a map $i_Y^\cal M \circ \sigma_X^\cal M \circ (i_X^\cal M)^{-1}: H_+^*(X,\R) \to H_+^*(Y,\R)$, and hence, after normalization, a map from the set of states for $X$ to the set of states for $Y$. The natural question arises, whether this map can be defined alternatively by a more direct approach.

Indeed, if $\sigma: \cal A^* \to \cal B^*$ is letter-to-letter, then for two subshifts $X \subset \cal A^\Z$ and $Y = \sigma(X) = \sigma^\Z(X)$ any function $g \in \partial_T C(Y, \Z)$ defines a function $g \circ \sigma^\Z \in C(X, \Z)$ which also satisfies
\begin{equation}
\label{eq.coinvariants+}
g \circ \sigma^\Z \in \partial_T C(X, \Z)\, .
\end{equation}
Thus $\sigma$ defines a contravariant morphism $H(\sigma): H(Y, \Z) \to H(X, \Z)$, and hence an induced dual map $H(\sigma)^*: H(X, \R)^* \to H(Y, \R)^*$. 
Since $X$ is assumed to be recurrent, 
we obtain via the above isomorphisms $i_X^\cal M$ and $i_Y^\cal M$ a map $(i_Y^\cal M)^{-1} \circ H(\sigma)^* \circ i_X^\cal M : \cal M(X) \to \cal M(Y)$, which agrees 
with the map $\sigma^\cal M$, since in this letter-to-letter case both coincide 
with the push-forward map $\sigma^\Z_*\, $.

However, for any attempt to
push the above approach further to get an alternative ``coinvariant'' description of the measure transfer map,
there is a problem at the very base, in that the property (\ref{eq.coinvariants+}) fails to hold in most cases where $\sigma$ is not letter-to-letter. We describe below a simple example which highlights the basic problem:

\medskip

Set $\cal A = \{a, b, c\}$ and $B = \{d, e\}$, and define $\sigma: \cal A^* \to \cal B^*$ by setting $\sigma(a) = d\,  , \,\, \sigma(b) = e$ and $\sigma(c) = de$. Set $X = \cal A^\Z$ and $Y = \cal B^\Z$ and notice $Y = \sigma(X)$.

Consider now the function $g: Y \to \Z$ defined by $g({\bf y}) = 0$ if ${\bf y} \in [d]$ and  $g({\bf y}) = 1$ if ${\bf y} \in [e]$, which gives $g({\bf y}) = 0$ if ${\bf y} \in [dd] \cup [de]$ as well as $g({\bf y}) = 1$ if ${\bf y} \in [ed] \cup [ee]$. Since $T([dd] \cup [ed]) = [d]$ and $T([de] \cup [ee]) = [e]$, we obtain $g \circ T ({\bf y}) = 0$ if ${\bf y} \in [dd] \cup [ed]$ and $g \circ T ({\bf y}) = 1$ if ${\bf y} \in [de] \cup [ee]$. Hence the function $g \circ T - g \in \partial_T C(Y,\Z)$ satisfies $(g \circ T - g) ({\bf y}) = 0$ if ${\bf y} \in [dd] \cup [ee]$, $(g \circ T - g) ({\bf y}) = 1$ if ${\bf y} \in [de]$ and $(g \circ T - g) ({\bf y}) = -1$ if ${\bf y} \in [ed]$.

Let us now consider the function $(g \circ T - g) \circ \sigma^\Z \in C(X, \Z)$, and its value on the periodic word ${\bf x} \in X$ given by ${\bf x} = \ldots ccc \ldots$. We have ${\bf x}_{[1, +\infty)} = ccc \ldots$ and thus $\sigma^\Z({\bf x})_{[1, +\infty)} = \sigma({\bf x}_{[1, +\infty)}) = dedede \ldots$, so that $\sigma^\Z({\bf x}) \in [de]$. As a consequence we obtain $(g \circ T - g) \circ \sigma^\Z({\bf x}) = 1$, and since $T({\bf x}) = {\bf x}$, also $(g \circ T - g) \circ\sigma^\Z \circ T^n({\bf x}) = 1$ for any integer $n \geq 0$. However, from Proposition 3.2.1 of \cite{Durand-Perrin} we know that for any function $f \in \partial_T(C, \Z)$ the sequence of function $f + f \circ T + \ldots + f \circ T^n$ is bounded uniformly. It follows that $(g \circ T - g) \circ \sigma^\Z \notin \partial_TC(X, \Z)$, so that property (\ref{eq.coinvariants+}) fails.
\end{rem}

\section{Evaluation of the transferred measure $\sigma^\cal M(\mu)$}
\label{sec:4d}

\subsection{A first example for the measure transfer}

${}^{}$

\smallskip

We will illustrate in this subsection the induced measure transfer map $\sigma^\cal M$ from Definition-Remark \ref{3.5} for an explicitely given morphism $\sigma: \cal A^* \to \cal B^*$. For this example 
we will carry through in all detail, for any invariant measure $\mu$ on $\cal A^\Z$, the computation of the values of $\sigma^\cal M(w)$ for any word $w \in \cal A^*$ of length $|w| \leq 2$.

\begin{convention}
\label{4.1d}
We simplify here (as well as in 
some 
other concrete computations below) the notation used before by writing $a, b, c, \ldots$ instead of $a_1, a_2, a_3, \ldots$ for the elements of the given alphabet $\cal A$, and we write $a_k$ or $b_k$ instead of $a(k)$ or $b(k)$ for the letters of a corresponding subdivision alphabet.
\end{convention}

\smallskip

For the alphabets $\cal A = \{a, b\}$ and $\cal B = \{c, d\}$ 
consider the morphism given by
$$\sigma: \cal A^* \to \cal B^* \, , \,\,\, a \mapsto cdc,\, b \mapsto dcc \, .$$ 
We derive $\ell_\sigma(a) = \ell_\sigma(b) = 3$ and $\cal A_\sigma = \{a_1, a_2, a_3, b_1, b_2, b_3\}$ as well as the corresponding subdivision morphisms $\pi_\sigma: \cal A^* \to \cal A_\sigma^*$ given by
$$\pi_\sigma(a) = a_1 a_2 a_3, \, \pi_\sigma(b) = b_1 b_2 b_3 \, .$$
Similarly, the corresponding letter-to-letter morphism $\alpha_\sigma: \cal A_\sigma^* \to \cal B^*$ is given by
$$
\alpha_\sigma(a_1) = \alpha_\sigma(a_3) = \alpha_\sigma(b_2) = \alpha_\sigma(b_3) = c, \,\,\alpha_\sigma(a_2) = \alpha_\sigma(b_1) = d \, .$$
We obtain (for $\mu_\sigma := \mu_{\ell_\sigma}$)
$$\mu_\sigma(a_1) = \mu_\sigma(a_2) = \mu_\sigma(a_3) = \mu(a),\,\, \mu_\sigma(b_1) = \mu_\sigma(b_2) = \mu_\sigma(b_3) = \mu(b) \, ,$$
and thus (following 
(\ref{eq3.2}))
$$
\begin{array}{ccccc}
(\alpha_\sigma)_*(\mu_\sigma)(c) &= &\mu_\sigma(a_1) + \mu_\sigma(a_3) + \mu_\sigma(b_2) + \mu_\sigma(b_3) &=& 2 \mu(a) + 2 \mu(b) \, ,
\\
(\alpha_\sigma)_*(\mu_\sigma)(d) &=&\mu_\sigma(a_2) + \mu_\sigma(b_1)
&=& \mu(a) + \mu(b)  \, .
\end{array}
$$

Similarly, one computes
$$
\begin{array}{c}
\mu_\sigma(a_1 a_2) = \mu_\sigma(a_2 a_3) = \mu(a),\,\, \mu_\sigma(b_1 b_2) = \mu_\sigma(b_2 b_3) = \mu(b)\, ,\\
\mu_\sigma(a_3 a_1) = \mu(a a)\, ,
\,\, \mu_\sigma(a_3 b_1) = \mu(a b)\, ,
\,\, \mu_\sigma(b_3 a_1) = \mu(b a)\, ,
\,\, \mu_\sigma(b_3 b_1) = \mu(b b) \, .
\end{array}
$$
and 
$\mu_\sigma(w) = 0$
for any other $w \in \cal A^*$ of length $|w| = 2$.

Correspondingly, one obtains 
$$
\begin{array}{clc}
(\alpha_\sigma)_*(\mu_\sigma)(cc) 
= 
&[\mu_\sigma(a_1 a_1) + \mu_\sigma(a_1 a_3) + \mu_\sigma(a_1 b_2) + \mu_\sigma(a_1 b_3) ]
\\
&+
[\mu_\sigma(a_3 a_1) + \mu_\sigma(a_3 a_3) + \mu_\sigma(a_3 b_2) + \mu_\sigma(a_3 b_3) ]
\\
&+[\mu_\sigma(b_2 a_1) + 
\mu_\sigma(b_2 a_3) + \mu_\sigma(b_2 b_2) + \mu_\sigma(b_2 b_3) ]
\\
&+
[\mu_\sigma(b_3 a_1) + \mu_\sigma(b_3 a_3) + \mu_\sigma(b_3 b_2) + \mu_\sigma(b_3 b_3)] 
\\
\qquad \qquad \qquad \,\,\, =& \mu(a a) + \mu(b) + \mu(b a)\, ,
\end{array}
$$
further
$$
\begin{array}{c}
(\alpha_\sigma)_*(\mu_\sigma)(cd) 
= 
[\mu_\sigma(a_1 a_2) + \mu_\sigma(a_1 b_1) ]
+
[\mu_\sigma(a_3 a_2) + \mu_\sigma(a_3 b_1) ]
\\
+
[\mu_\sigma(b_2 a_2) + \mu_\sigma(b_2 b_1) ]
+
[\mu_\sigma(b_3 a_2) + \mu_\sigma(b_3 b_1) ]
\\
=
\mu(a) + \mu(ab) + 
\mu(bb)\, ,
\end{array}
$$
$$
\begin{array}{c}
(\alpha_\sigma)_*(\mu_\sigma)(dc) 
= 
[\mu_\sigma(a_2 a_1) + \mu_\sigma(a_2 a_3)
+
\mu_\sigma(a_2 b_2) + \mu_\sigma(a_2 b_3) ]
\\
+
[\mu_\sigma(b_1 a_1) + \mu_\sigma(b_1 a_3)
+
\mu_\sigma(b_1 b_2) + \mu_\sigma(b_1 b_3) ]
\\
= \mu(a) + \mu(b)\, ,
\end{array}
$$
and finally
$$
\begin{array}{c}
(\alpha_\sigma)_*(\mu_\sigma)(dd) 
= 
[\mu_\sigma(a_2 a_2) + \mu_\sigma(a_2 b_1)]
+
[\mu_\sigma(b_1 a_2) + \mu_\sigma(b_1 b_1)]
= 0\, .
\end{array}
$$

Since by definition we have $\sigma^\cal M(\mu) = (\alpha_\sigma)_*(\mu_\sigma)$, we have computed
$$
\begin{array}{cclc}
\sigma^\cal M(\mu)(c) &=& 2(\mu(a) + \mu(b))
\\
\sigma^\cal M(\mu)(d) &=& \mu(a) + \mu(b)
\\
\sigma^\cal M(\mu)(cc) &=& \mu(b) + \mu(aa) + \mu(ba)
\\
\sigma^\cal M(\mu)(cd) &=& \mu(a) + \mu(ab) 
+ \mu(bb)
\\
\sigma^\cal M(\mu)(dc) &=& \mu(a) + \mu(b) 
\\
\sigma^\cal M(\mu)(dd) &=& 0 \, .
\end{array}$$  

\subsection{An alternative evaluation method}

${}^{}$

\smallskip

As illustrated by the example considered in the previous subsection, already for fairly simple morphisms $\sigma$ the preimage set $\alpha_\sigma^{-1}(w)$ may become rather large, even for small $|w|$.
In this subsection we explain how a more efficient evaluation technique is obtained (compare formulas (\ref{image-measure2}) and (\ref{image-measure})) and we give an example of a typical computation.

\smallskip

Given a morphism $\sigma: \cal A^* \to \cal B^*$, consider any word $w = x_1 \ldots x_n \in \cal A^*$ and its image $\sigma(w) = y_1 \ldots y_m \in \cal B^*$, with letters $x_i \in \cal A$ and $y_j \in \cal B$. 
An {\em occurrence of 
$w' \in \cal B^*$ 
in $\sigma(w)$} is a factor $y_{k'} \ldots y_{\ell'}$ of $y_1 \ldots y_m$ which satisfies $w' = y_{k'} \ldots y_{\ell'} \, .$

An occurrence of $w'$ in $\sigma(w)$ is called {\em essential} if its first letter occurs in $\sigma(x_1)$ and its last letter in $\sigma(x_n)$. 
In particular, for $|w| = 1$ any occurrence of $w'$ in $\sigma(w)$ is {essential}. If $w$ has length $|w| =  2$, then an occurrence of $w'$ in $\sigma(w)$ is {essential} if the factor $w'$ overlaps from $\sigma(x_1)$ into $\sigma(x_2)$. 
For $|w| \geq 3$ a factor $w'$ of $\sigma(w)$ is an essential occurrence if $w'$ contains the image $\sigma(x_2 \ldots x_{n-1})$ of the ``maximal inner factor'' $x_2 \ldots x_{n-1}$ of $w$ as factor, but not as prefix or suffix.

The number of all essential occurrences of $w'$ in $\sigma(w)$ will be denoted by 
$\lfloor \sigma(w)\rfloor_{w'}$. It follows directly from these definitions that for any $w' \in \cal B^*$ the set of all words $w$ in $\cal A^*$, for which $\sigma(w)$ contains at least one essential occurrence of $w'$, is finite. Indeed, 
for $|w'| \geq 2$ 
one easily verifies that 
any such $w$ must satisfy 
\begin{equation}
\label{bound}
\frac{|w'|}{||\sigma||} \leq |w| \leq 
\frac{|w'|-2}{\langle \sigma \rangle}+2 \, ,
\end{equation}
where 
$||\sigma||$ and $\langle \sigma \rangle$ 
denote the biggest and smallest length 
respectively 
of any of the letter images $\sigma(a_i)$. 

\begin{prop}
\label{3.5.5}
Let $\sigma: \cal A^* \to \cal B^*$ be any non-erasing morphism of free monoids, and let $\mu$ be any invariant measure on $\cal A^\Z$. Then for any $w' \in \cal B^*$ the transferred measure $\mu^\sigma = \sigma^\cal M(\mu)$, evaluated on the cylinder $[w']$, has the value
\begin{equation}
\label{image-measure1.5}
\mu^\sigma(w') = 
\sum_{u \,\in\, \cal A^* }
{\lfloor\sigma(u) \rfloor}_{w'} \cdot \mu(u) \, .
\end{equation}
In particular, for $|w'| \geq 2$, we have
\begin{equation}
\label{image-measure2}
\mu^\sigma(w') = 
\sum_{{\big \{}u \,\in\, \cal A^* \,{\big |}\, |u| \leq \frac{|w'|-2}{\langle \sigma\rangle}+2 {\big \}}}
{\lfloor\sigma(u) \rfloor}_{w'} \cdot \mu(u) \, .
\end{equation}
\end{prop}

\begin{proof}
Recall 
that in formula (\ref{image-measure}), for any $u \in \alpha_\sigma^{-1}(w')$, 
either $\hat u$ is defined as shortest word in $\cal A^*$ with the property that $\pi_\sigma(\hat u)$ contains $u$ as factor, or else (if such $\hat u$ doesn't exist) we have formally set $\mu(\hat u) = \mu_{\ell_\sigma}(u) = 0$.
In the first case 
the factor $u$ of $\pi_\sigma(\hat u)$ defines an essential occurrence of $w' = \alpha_\sigma(u)$ in $\sigma(\hat u)$.  

Conversely, every essential occurrence of $w'$ in 
$\sigma(v)$, 
for any $v \in \cal A^*$, defines a factor $u$ of $\pi_\sigma(v)$ with $u \in \alpha_\sigma^{-1}(w')$ for which we have $\hat u = v\,$. Indeed, the 
word $\hat u$ can not be a proper factor of $v$, or else the given occurrence of $w'$ as factor 
of $\sigma(v)$ would not have been essential.

It follows that the sums in the formulas (\ref{image-measure}) and (\ref{image-measure2}) differ only in their organization of the indexing, 
so that the results of the right hand sides of (\ref{image-measure}) and of (\ref{image-measure2}) must be equal.
\end{proof}

We will illustrate now that the new formula (\ref{image-measure2}) is a lot more convenient in practice (in particular in view of the fact that in the example below the set $\alpha_\sigma^{-1}(w)$ consists of 
648 elements):

\begin{example}
\label{autre-exemple}
Let us consider 
$\cal A = \{a, b, c\}, \, \cal B = \{d, e\}$ and $\sigma$ given by
$$a \mapsto ded, \, b \mapsto de, \, c \mapsto dedd \, .$$
Let us compute $\mu^\sigma(w) \,[= \sigma^\cal M(\mu)(w)]$ for any invariant measure $\mu$ on $\cal A^\Z$, for the 
randomly picked word 
$$w = dded \, .$$
By (\ref{bound}) it suffices to consider any 
$u \in \cal A^*$ of length $|u |\leq 3$. We 
quickly check that that 
$$\lfloor\sigma(u)\rfloor_{ w} = 0 \quad \text{for} \quad
u \in \{a, b, c, ab, ba, bb, bc, cb\}$$
and 
$$\lfloor\sigma(u)\rfloor_{ w} = 1 \quad \text{for} \quad
u \in \{aa, ac, ca, cc\}\, .$$
For any word 
$u = x_1 x_2 x_3 \in \cal A^*$ with $x_1, x_2, x_3 \in \cal A$ and $x_2 \neq b$ 
the definition of $\lfloor\sigma(u)\rfloor_{w}$ gives directly $\lfloor\sigma(u)\rfloor_{w} = 0$.
It remains to 
check that
$$\lfloor\sigma(u)\rfloor_{ w} = 0 \quad \text{for} \quad
u \in \{bba, bbb, bbc\}\, ,$$ 
and 
$$\lfloor\sigma(u)\rfloor_{ w} = 1 \quad \text{for} \quad
u \in 
\{aba, abb, abc, cba, cbb, cbc\}\, ,$$
in order to obtained the desired formula
$$\mu^\sigma(w) = \mu(aa) + \mu(ac) + 
\mu(ca) + \mu(cc) + 
\mu(aba) + \mu(abb) + \mu(abc) + \mu(cba) + \mu(cbb) + \mu(cbc)$$
$$= 
\mu(aa) + \mu(ac) + \mu(ca) + \mu(cc) + 
\mu(ab) + \mu(cb)
=
\mu(a) + \mu(c) \, .
$$
\end{example}

\begin{rem}
\label{letter-image-frequency}
Since a factor $w' \in \cal B^*$ of length $|w'| = 1$ in any $\sigma(w) = \sigma(x_1 \ldots x_n)$ cannot ``overlap'' from some $\sigma(x_i)$ into the adjacent $\sigma(x_{i+1})$, 
the only essential occurrences of $w'$ in any $\sigma(w)$ can take place if one has $|w| = 1$.
Hence we deduce from (\ref{image-measure1.5}) for any $b_j \in \cal B$ the formula
\begin{equation}
\label{image-measure3}
\mu^\sigma(b_j) \, [=\sigma^\cal M(\mu)(b_j)] = 
\sum_{a_k \in \cal A}
{|\sigma(a_k) |}_{ b_j } \cdot \mu(a_k) \, .
\end{equation}
\end{rem}

Every shift-invariant measure $\mu$ on $\cal A^\Z$ defines canonically a 
{\em letter frequency vector} 
$\vec v(\mu) = (\mu([a_k]))_{a_k \in \cal A} \in \R_{\geq 0}^\cal A$, 
which plays an important role in many contexts (see 
for instance 
\cite{BHL2.8-II} or \cite{B+Co}). 

We observe directly from (\ref{image-measure3}):

\begin{prop}
\label{22.7}
Let $\sigma: \cal A^* \to \cal B^*$ be a non-erasing monoid morphism, and let $\mu \in \cal M(\cal A^\Z)$ be an invariant measure. Then the letter frequency vectors $\vec v(\mu)$ and 
$\vec v(\mu^\sigma)$, associated to $\mu$ and to its transferred measure $\mu^\sigma = \sigma^\cal M(\mu) \in \cal M(\cal B^\Z)$ respectively, satisfy
\begin{equation}
\label{more-formally}
\vec v(\mu^\sigma) = M(\sigma) \cdot \vec v(\mu) \, ,
\end{equation}
where $M(\sigma)$ denotes the incidence matrix of $\sigma$
- see
(\ref{incidence-m}).
\qed
\end{prop}


\section{Shift-orbit injectivity and related notions}
\label{sec:5d}

In this section we will establish a natural criterion which guaranties that the measure transfer map $\sigma^\cal M$ is $1-1$, when restricted to measures which are supported by suitable subshifts. We first need to recall and specify the notation from Section \ref{sec:2.3}:  

\begin{defn}
\label{5n.1}
Let $\sigma: \cal A^* \to \cal B^*$ be a non-erasing monoid morphism, and let $X \subset \cal A^\Z$ be any subshift.
\begin{enumerate}
\item
We say that $\sigma$ is {\em shift-orbit injective in $X$} if the map $\sigma^T$ restricted to the shift-orbits of $X$ is injective. 
\item
We say that $\sigma$ is {\em shift-period preserving in $X$} if $\sigma$ preserves the shift-period for every periodic orbit in $X$ (see Definition \ref{period-preserving}).

\end{enumerate}
\end{defn}

\begin{lem}
\label{5.x.d}
Let $\sigma': \cal A^* \to \cal B^*$ and $\sigma'': \cal B^* \to \cal C^*$ be two non-erasing morphisms
(so that the composition 
$\sigma:= \sigma'' \circ \sigma'$ is 
also non-erasing.) 
For any subshift $X \subset \cal A^\Z$ consider 
the image subshift $Y = \sigma(X)$.
Then we have:
\begin{enumerate}
\item
The map 
$\sigma$ is shift-orbit injective in $X$ if and only if $\sigma'$ is shift-orbit injective in $X$ and $\sigma''$ is shift-orbit injective in $Y$.
\item
The map 
$\sigma$ is shift-period preserving in $X$ if and only if $\sigma'$ is shift-period preserving in $X$ and $\sigma''$ is shift-period preserving in $Y$.

\end{enumerate}
\end{lem}

\begin{proof}
Using the surjectivity from Lemma \ref{closure-not} (2), we obtain the first of these statements as direct application of the fact that for any two composable surjective maps $f$ and $g$ the composition $g \circ f$ is injective if and only if both maps $f$ and $g$ are injective. 
Statement (2) follows directly from the 
trivial 
observation that $\sigma(w)$ is a proper power (see (\ref{eq:prop-pow})) if and only if one of the three, $w, \sigma'(w)$ or $\sigma''(\sigma'(w))$ is a proper power.
\end{proof}

Next we consider a 
subdivision length function $\ell$ and the associated subdivision morphism $\pi_\ell: \cal A^* \to \cal A_{\ell}^*$ as in Subsection \ref{sec:subdivision-morphs}. 

\begin{lem}
\label{5n.2}
Any 
subdivision morphism $\pi_\ell: \cal A^* \to \cal A_\ell^*$ is both, shift-orbit injective and shift-period preserving in the full shift $\cal A^\Z$.
\end{lem}

\begin{proof}
In order to see that $\pi_\ell$ is shift-orbit injective in the full shift we only need to observe that $\pi_\ell^\Z: \cal A^\Z \to \cal A_\ell^\Z$ is injective, and that two biinfinite words from $\pi_\ell^\Z(\cal A^\Z)$ belong to the same shift-orbit if and only if their preimages in $\cal A^\Z$ belong to the same orbit.

Similarly, from the definition of $\pi_\ell$ it follows directly that the image of any $w \in \cal A^*$ is a proper power if and only $w$ itself is a proper power. This implies directly (see Definition \ref{period-preserving}) that $\pi_\ell$ is shift-period preserving in the full shift.
\end{proof}

\begin{lem}
\label{4.12+}
For any subdivision morphism $\pi_{\ell}: \cal A^* \to \cal A_{\ell}^*$ 
the induced 
transfer map $\pi_\ell^\cal M: \cal M(\cal A^\Z) \to \cal M(\cal A_{\ell}^\Z)\, , \,\, \mu \mapsto \mu_{\ell}$ on the measure cones is injective.
\end{lem}

\begin{proof}
Two measures $\mu, \mu' \in \cal M(\cal A^\Z)$ are distinct if and only there exists a word $w \in \cal A^*$ where the 
associated weight functions satisfy $\mu(w) \neq \mu'(w)$. 
From the definition of the subdivision measure in Definition \ref{3.3} we see directly that $\mu_{\ell}(\pi_{\ell}(w)) = \mu(w)$ and $\mu'_{\ell}(\pi_{\ell}(w)) = \mu'(w)$. It follows that $\mu_{\ell} \neq \mu'_{\ell}$.
\end{proof}

We can now observe:

\begin{lem}
\label{ergodic-new}
Let $\sigma: \cal A^* \to \cal B^*$ be a non-erasing monoid morphism, let $X \subset \cal A^\Z$ be any subshift, and let 
$\mu \in \cal M(X)$ be an ergodic measure. Then $\mu^\sigma$ is an ergodic measure on the image subshift $\sigma(X)$.
\end{lem}

\begin{proof}
As before, we write $\sigma$ as composition $\sigma =\alpha_\sigma \circ \pi_\sigma$ of the subdivision morphism $\pi_\sigma$ and the letter-to-letter morphism $\alpha_\sigma\, $. From Lemma \ref{4.12+} we know that the induced map $\pi_\sigma^\cal M: \cal M(X) \to \cal M(\pi_\sigma(X))$ is injective and hence, by Lemma \ref{3.6} (a) and the surjectivity from Proposition 4.4 of \cite{BHL2.8-II}, an $\R_{\geq 0}$-linear isomorphism of cones. It follows that any extremal point of the cone $\cal M(X)$ is mapped to an extremal point of the cone $\cal M(\pi_\sigma(X))$. But a measure is ergodic if and only if it is an extremal point of the corresponding measure cone, so that, for any ergodic measure $\mu \in \cal M(X)$ the transferred  measure $\mu^{\pi_\sigma}$ is also ergodic.

We now use the well known - and easy to prove - fact (see Proposition 7.9 of \cite{Tatjana}), that for any ergodic measure the push-forward measure with respect to any factor map (such as our letter-to-letter morphism $\alpha_\sigma$) is again ergodic, in order to conclude that  the measure $(\mu^{\pi_\sigma})^{\alpha_\sigma} = \mu^\sigma$ is ergodic.
\end{proof}

We can now start the proof of the main result of this paper; its core (Theorem \ref{8n.5}) will however be delayed until the next section.

\begin{thm}
\label{5n.3}
Let $\sigma: \cal A^* \to \cal B^*$ be a non-erasing morphism of free monoids on finite alphabets, and let $X \subset \cal A^\Z$ be a subshift, 
with image subshift $\sigma(X)$.

If $\sigma$ is shift-orbit injective in $X$, then the measure transfer map $\sigma_X^\cal M: \cal M(X) \to \cal M(\sigma(X))$ is injective.
\end{thm}

\begin{proof}
We consider the canonical decomposition $\sigma = \alpha_\sigma \circ \pi_\sigma$ from equality (\ref{eq:can-dec}) and obtain from the functoriality of the measure transfer (see Lemma \ref{3.6} (b)) the decomposition $\sigma_X^\cal M = (\alpha_\sigma)_{\sigma(X)}^\cal M \circ (\pi_\sigma)_X^\cal M$. 

From Lemma \ref{5.x.d} (1) and Lemma \ref{5n.2} we obtain directly that the morphism $\alpha_\sigma$ is shift-orbit injective. We can thus apply Lemma \ref{4.12+} to $(\pi_\sigma)_X^\cal M$ and Theorem \ref{8n.5} to $(\alpha_\sigma)_{\sigma(X)}^\cal M$ to deduce that $\sigma_X^\cal M$ is injective.
\end{proof}

We will terminate this section with a discussion that compares the above introduced notions to the more frequently used notions of morphisms that are ``recognizable'' or a ``recognizable for aperiodic points" (see Fig.1). For the convenience of the reader we briefly recall the definitions:

\begin{defn}
\label{5n.recog}
Let $\sigma:\cal A^* \to \cal B^*$ be a non-erasing morphism, and let $X \subset \cal A^\Z$ be a subshift over $\cal A$. 

\smallskip
\noindent
(1)
Then $\sigma$ is 
said\,\footnote{ 
Some authors say ``recognizable on $X$''.} to be 
{\em recognizable in $X$} if the following conclusion is true:

\noindent
Consider biinfinite words ${\bf x}, {\bf x'} \in X \subset \cal A^\Z$ and ${\bf y} \in \cal B^\Z$ which satisfy  
\begin{enumerate}
\item[(*)]
${\bf y} = T^k(\sigma^\Z({\bf x}))$ and ${\bf y} = T^\ell(\sigma^\Z({\bf x'}))$ for some integers $k, \ell$ which satisfy $0 \leq k \leq |\sigma(x_1)|-1$  and $ 0 \leq  \ell \leq |\sigma(x'_1)|-1$, where $x_1$ and $x'_1$ are the first letters of the 
positive half-words 
${\bf x_{[1, \infty)}} = x_1  x_2 \ldots$ of ${\bf x}$ and ${\bf x'_{[1, \infty)}} = x'_1 x'_2 \ldots$ of ${\bf x'}$ respectively.
\end{enumerate}
Then one has ${\bf x} = \bf x'$ and $k = \ell$.

\smallskip
\noindent
(2)
The morphism $\sigma$ is called {\em recognizable for aperiodic points in $X$} if the same conclusion as in (1) is true, but under the strengthened hypothesis that in addition $\bf y$ is assumed not to be a periodic word.
\end{defn}

\begin{warning}
\label{Warning}
The terminology ``recognizable for aperiodic points in $X$'' should not be misunderstood as to be referring to aperiodic words in $X$: it really does concern aperiodic words in 
the image subshift $\sigma(X)$. 
The alternative wording ``recognizable in $X$ for aperiodic points of $\sigma(X)$'' would be more accurate, but we prefer here to stick to the established terminology in the literature.
\end{warning}

\begin{figure}

\begin{center}
\begin{tikzpicture}

\node[rounded corners=6pt,fill=gray!10] (R) at (0,-2.5) {recognizable};
\node[rounded corners=6pt,fill=gray!10] (RA) at (6,-2.5) {$\begin{array}{c}
\text{recognizable} \\
\text{for aperiodic points}
\end{array}$};
\node[rounded corners=6pt,fill=gray!10] (O) at (-.5,.1) 
{$\begin{array}{c}
\text{shift-orbit injective} \\
\text{and} \\
\text{shift-period preserving}
\end{array}$};
\node[rounded corners=6pt,fill=gray!10] (OP) at (6,0) 
{shift-orbit injective};

\node[rounded corners=6pt,fill=gray!10] (RA) at (12,-1) {$\begin{array}{c}
\text{measure transfer} \\
\text{injective}
\end{array}$};

\node[rotate=90] at (0,-1.3) {$\Longleftrightarrow$};

\node at (2.7,-0.2) {$\Longrightarrow$};
\node[rotate=180] at (2.7,0.2) {$\Longrightarrow$};
\node[rotate=180] at (2.7,0.2) {$\!/$};

\node at (2.7,-2.7) {$\Longrightarrow$};
\node[rotate=180] at (2.7,-2.3) {$\Longrightarrow$};
\node[rotate=180] at (2.7,-2.3) {$\!/$};

\node[rotate=-90] at (5.8,-1.3) {$\Longrightarrow$};
\node[rotate=90] at (6.2,-1.3) {$\Longrightarrow$};
\node[rotate=90] at (6.2,-1.3) {$\!/$};

\node[rotate=-25] at (9,-0.4) {$\Longrightarrow$};
\node[rotate=155] at (8.9,-0.7) {$\Longrightarrow$};
\node[rotate=155] at (8.9,-0.7) {$\!/$};

\node[rotate=25] at (9,-2.1) {$\Longrightarrow$};
\node[rotate=25] at (9,-2.1) {$\!/$};
\node[rotate=205] at (9.1,-2.4) {$\Longrightarrow$};
\node[rotate=205] at (9.1,-2.4) {$\!/$};

\end{tikzpicture}
\end{center}
\caption{}
\end{figure}


It turns out (see 
Proposition 3.8 of \cite{BHL2.8-II}) that for any non-erasing morphism 
$\sigma$ as above the condition ``recognizable in a subshift $X$'' is equivalent to the condition ``shift-orbit injective and shift-period preserving in $X$''. Indeed, a quick proof for the injectivity of the measure transfer map $\sigma_X^\cal M$ as in Theorem \ref{5n.3}, under the stronger hypothesis of ``recognizability in $X$'' is given in section 
3.3 of \cite{BHL2.8-II}. 

However, since recognizability can 
only be achieved for an everywhere growing $S$-adic development of a given subshift $X$ (for the terminology see for instance Section 2 of \cite{BHL2.8-II}) if $X$ is 
aperiodic, the 
much more popular hypothesis used by the $S$-adic community is not ``recognizable'' but ``recognizable for aperiodic points'', or ``eventually recognizable for aperiodic points''. 
The relation of this notion to the concepts introduced in this section 
(see Fig. 1) is given by the following; a formal proof 
is given in Remark 
3.10 (3) of \cite{BHL2.8-II}.

\begin{prop}
\label{6.3}
If a non-erasing morphism $\sigma:\cal A^* \to \cal B^*$ is shift-orbit injective in a subshift $X \subset \cal A^\Z$, then $\sigma$ is recognizable for aperiodic points in $X$.
\qed
\end{prop}

The converse implication of Proposition \ref{6.3}
turns however 
out to be wrong: 
A very simple counterexample is given by the letter-to-letter morphism 
$\sigma_1: \{a, b\}^* \to \{c\}^*$ from Remark \ref{2.7d} and the subshift $X = \{a^{\pm \infty}, b^{\pm \infty}\}$.
Since $\{c\}^*$ consists only of periodic words, it follows (see Warning~\ref{Warning}) that $\sigma_1$ is automatically recognizable for aperiodic points in $X$. However, since $\sigma_1^T(\cal O( a^{\pm \infty})) = \cal O( c^{\pm \infty} )$ and $\sigma_1^T(\cal O( b^{\pm \infty} )) = \cal O( c^{\pm \infty} )$, the map $\sigma_1$ is not shift-orbit injective in $X$.

\smallskip

This example shows also that ``$\sigma$ recognizable for aperiodic points in $X$'' is in general too weak to be able to deduce that the measure transfer map $\sigma_X^\cal M : \cal M(X) \to \cal M(\sigma(X))$ is injective: For $\sigma_1$ as above one observes directly from Lemma \ref{3.6} (d) that $\sigma_1^\cal M(\mu_a) = \sigma_1^\cal M(\mu_b) = \mu_c$. 

We also observe that the injectivity of the measure transfer map $\sigma_X^\cal M$ does in general not imply the injectivity of the map $\sigma_X^T$ induced by $\sigma$ on the orbits of $X$, so that the converse of Theorem \ref{5n.3} does not hold in full generality. Indeed, we have:

\begin{rem}
\label{wrong-converse}
Let  $X \subset \{a, b\}^*$ be the countable subshift which consists of the two orbits defined by
$$\ldots aaa \ldots \qquad 
\text{and} \qquad
\ldots aaabaaa\ldots \, .$$
Set $ Y = \{c\}^*$, and consider again the morphism $\sigma_1$ from Remark \ref{2.7d}, given by $a \mapsto c$ and $b \mapsto c$. This map is not injective on the shift-orbits of $X$, but the map $(\sigma_1)_X^\cal M: \cal M(X) \to \cal M(Y)$ is injective, since
the only non-trivial invariant probability measure on $X$ is the atomic measure $\mu_a$ concentrated on the element $\ldots aaa \ldots$. This can be seen from a direct application of the Kirchhoff rules (\ref{Kirchhoff}). Alternatively one can use the fact that $X$ is the substitutive subshift associated to the substitution given by $a \mapsto a^2\, , \,\, b \mapsto aba$ and apply Corollary 3.5 (1) of \cite{BHL2}.
\end{rem}

More generally, if a subshift $X \subset \cal A^\Z$ contains a non-empty subset $X_0 \subset X$ which consists entirely of biinfinite words that lie outside the support of any invariant measure on $X$, then the violation of the injectivity of $\sigma^T$ on $X_0$ will not have any impact on the injectivity of the measure transfer map $\sigma^\cal M$, for any morphism $\sigma: \cal A^* \to \cal B^*$.
This is the basic principle which is now used in the proof of the following proposition, which shows that the injectivity of a map $\sigma_X^\cal M$ as above is in general too weak to deduce that $\sigma$ is recognizable for aperiodic points. It completes the set of implications and their refusals between the properties discussed here, as summarized in Fig. 1.

\begin{prop}
\label{last-implication}
There exist a subshift $X \subset \cal A^\Z$ and a morphism $\sigma: \cal A^* \to \cal B^*$ which is not recognizable for aperiodic points in $X$, while the induced measure transfer map $\sigma_X^\cal M: \cal M(X) \to \cal M(\sigma(X))$ is injective.
\end{prop}

\begin{proof}
We will give a concrete example of $X$ and $\sigma$ as claimed: Set $\cal A = \{a, b, c\}$ and $\cal B = \{d, e\}$, and let $\sigma: \cal A^* \to \cal B^*$ be the letter-to-letter morphism given by $\sigma(a) = d$ and $\sigma(b) = \sigma(c) = e$. Let now $X  \subset \cal A^*$ be the countable subshift which consists of the three orbits defined by
$$\ldots aaa \ldots \,\,\,\,\, ,
 \,\,\, \,\,
\ldots aaabaaa\ldots \qquad
\text{and} \qquad
\ldots aaacaaa\ldots \, .$$
As in Remark \ref{wrong-converse} the only invariant probability measure on $X$ is the characteristic measure $\mu_a\,$, so that the induced measure transfer map $\sigma_X^\cal M: \cal M(X) \to \cal M(\sigma(X))$ is automatically injective. On the other hand, the two biinfninite indexed words $\ldots aaa \cdot baaa\ldots $ and  $\ldots aaa \cdot caaa\ldots $ lie in distinct shift-orbits and have the same image  $\ldots ddd \cdot eddd\ldots \,$, which is aperiodic. It follows that $\sigma$ is not recognizable for aperiodic points in $X$. 
\end{proof}

\begin{rem}
\label{after-Arnaud}
There exist also stronger but slightly more intricate examples as given in the last proof, where the image subshift $\sigma(X)$ is completely aperiodic. They are based on the same construction principle, but the periodic orbit of $\ldots aaa \ldots $ is replaced by the minimal and uniquely ergodic subshift $X_\tau$ defined by any primitive substitution $\tau: \cal A^* \to \cal A^*$, which in addition we assume to possess a left-infinite word $V$ and a right-infinite word $W$ that are both fixed by $\tau$. Such a substitution is given for example by the square of the Fibonacci substitution $\tau(x) = xyx\,,\,\, \tau(y) = xy$ with $V = \ldots \tau^2(xy) \, \tau(xy) \, xy \, x$ and $W = x\, yx \,\tau(yx)\, \tau^2(yx) \ldots$.

One now considers three additional letters $a, b, c$ not contained in $\cal A$ and defines the letter-to-letter morphism $\sigma: (\{b, c\} \cup \cal A)^* \to (\{a\} \cup \cal A)^*$ by setting $\sigma(a_i) = a_i$ for any $a_i \in \cal A$ and $\sigma(b) = \sigma(c) = a$. We define $X = X_\tau \cup \cal O(V b W) \cup \cal O(V c W)$ and observe that the biinfinite indexed words $V \cdot b W$ and $V \cdot c W$ are sent by $\sigma$ both to the same image word $V \cdot a W$ which is not periodic, so that $\sigma$ is not recognizable for aperiodic points in $X$. On the other hand, $X$ is easily seen to be the subshift $X = X_{\tau'}$ associated to the substitution $\tau': (\{b, c\} \cup \cal A)^* \to (\{b, c\} \cup \cal A)^*$ which coincides with $\tau$ on $\cal A$ and sends $b$ and $c$ to $v b w$ and $v c w$ respectively, for an arbitrary non-empty finite suffix $v$ of $V$ and prefix $w$ of $W$. It then follows from Corollary 3.5 (1) of \cite{BHL2} that $X$ is uniquely ergodic, so that $\sigma_X^\cal M$ is automatically injective.
\end{rem}

\begin{rem}
\label{after-referee}
(1)
From the examples in the proof of Proposition 5.10 and from Remark 5.11 one may get the impression that the only possibility, for a subshift $X$ and a morphism $\sigma$  which is not recognizable for aperiodic points in $X$ while its associated measure transfer map $\sigma^\cal M_X$ is injective, can occur if there is no invariant measure $\mu \in\cal M(X)$ with $\Supp(\mu) = X$. This impression, however, is treacherous: In Example 4.3 of \cite{BSTY19} a uniquely ergodic infinite minimal subshift $X \subset \cal A^\Z$ (indeed the substitution subshift of a primitive substitution) is exhibited as well as a morphism $\sigma$ which is $2:1$ on shift-orbits throughout all of $X$.

\smallskip
\noindent
(2)
If one attempts to find additional conditions on $X$ and $\sigma$ which ensure that the two properties ``recognizable for aperiodic points'' and ``injectivity of the measure transfer map $\sigma_X^\cal M\,$'' become equivalent, one should be aware of the following:

If $\sigma$ is shift-orbit injective on $X$, then we have shown above that both properties are satisfied. If, on the other hand, two distinct shift-orbits $\cal O({\bf x}) \neq \cal O({\bf x'})$ of $X$ have the same $\sigma$-image, then we observe:
\begin{enumerate}
\item[(a)]
The first property is violated if and only if the image orbit $\cal O({\sigma(\bf x})) = \cal O(\sigma({\bf x'}))$ is not periodic. 
\item[(b)]
For the second property we consider the subshifts $\bar{\cal O({\bf x})}$ and $\bar{\cal O({\bf x'})}$ in $X$ which are generated by $\bf x$ and $\bf x'$ respectively. In order to produce a violation to this second property, we need that $\bar{\cal O({\bf x})} \neq \bar{\cal O({\bf x'})}$. If one has furthermore that their measure cones $\cal M(\bar{\cal O({\bf x})})$ and $\cal M(\bar{\cal O({\bf x'})})$ are distinct subcones of $\cal M(X)$, then indeed the second property is violated.
\end{enumerate}
\end{rem}

\begin{rem}
\label{special-subhifts}
To finish this section, let us again consider the totality of the implications and their refusals as presented in Fig.1, for subshifts on which we impose additional assumptions. As before, we always consider a morphism $\sigma: \cal A^* \to \cal B^*$ as well as subshifts $X \subset \cal A^\Z$ and $Y = \sigma(X) \subset \cal B^\Z$.
\begin{enumerate}
\item
If $X$ is minimal, then the four left-most properties in Fig.1 for $\sigma$ on $X$ are equivalent to each other, while the right-most property is a consequence, but the converse implication fails (as shown above in Remark \ref{after-referee} (1)).
\item
The same is true if the image subshift $\sigma(X)$ (and hence also $X$ itself) is aperiodic, since in that case the equivalence between
``$\sigma$ recognizable for aperiodic points in $X$'' and ``$\sigma$ recognizable in $X$'' comes for free.
\item
If $\sigma(X)$ is not aperiodic, one may add the assumption that $X$ has a dense orbit, or a dense positive half-orbit. 
However, neither of these assumptions suffices to turn any of the 6 refused implications in Fig.1 into an implication. 
To be specific, we point out examples to the following three implication-refusals; the other three follow by using the implications from Fig.1.
\begin{enumerate}
\item[(a)] ``shift-orbit injective'' $\,\,\Longrightarrow  \!\!\!\!\!\!\!\!\!/ \quad$  ``shift-orbit injective and shift-period preserving'':

\quad
The simplest example with dense positive half-orbit in $X$ is given by $X = \{a\}^\Z$ and $\sigma: \{a\}^* \to \{b\}^*$ with $\sigma(a) = b^2$. 

\quad
More interesting examples are given by taking $\tau: \cal A^* \to \cal A^*$ to be any primitive substitution on two or more letters, and set $X'$ to be the substitution subshift $X' = X_\tau$. One then ``perturbes'' $\tau$ into $\tau'$ by adding a new generator $a'$ to $\cal A$ to obtain $\cal A'$, and defines $\tau'(a_k) = \tau(a_k) a'$ for any $a_k \in \cal A$, as well as $\tau'(a') = a'^2$. Then the periodic orbit $a'^{\pm \infty}$ belongs to the substitution subshift $X := X_{\tau'}\,$, and any ${\bf x} \in X$ which is different from $a'^{\pm \infty}$ has positive half-orbit which is dense in $X$. Furthermore, $\tau'$ is shift-orbit injective on $X$, while the shift-period of $a'^{\pm \infty}$ is not preserved by $\tau'$.

\smallskip

\item[(b)] ``recognizable for aperiodic points'' $\,\, \Longrightarrow  \!\!\!\!\!\!\!\!\!/\quad$  ``measure transfer injective'':

\quad
We start with $\tau$ and $X'$ as above, but add now two letters $a'$ and $a''$ to $\cal A$ to get the alphabet $\cal A''$. We then define the substitution $\tau'': \cal A'' \to \cal A''$ by setting
$\tau''(a_k) = a'' \tau(a_k) a'$ for any $a_k \in \cal A$, as well as $\tau''(a') = a'^2$ and $\tau''(a'') = a''^2 $. 
Again, the periodic orbits $a'^{\pm \infty}$ and  $a''^{\pm \infty}$ belong to the substitution subshift $X := X_{\tau''}$, and any ${\bf x} \in X$ which is different from $a'^{\pm \infty}$ and $a''^{\pm \infty}$ has positive half-orbit which is dense in $X$. For $\cal A' = \cal A \cup \{a'\}$ we now define $\sigma: \cal A''^* \to \cal A'^*$ to be the identity on $\cal A^* \subset \cal A''^*$ and by setting $\sigma(a') = \sigma(a'') = a'$. Then $\sigma$ is recognizable for aperiodic points in $X$, and the two characteristic measures $\mu_{a'}$ and  $\mu_{a''}$ are distinct, but have the same image $\mu_{a'}$ under the measure transfer map $\sigma_X^\cal M$.

\smallskip

\item[(c)] ``measure transfer injective'' $\,\,\Longrightarrow  \!\!\!\!\!\!\!\!\!/\quad$  ``recognizable for aperiodic points'':

\quad
Here we can point to the infinite minimal substitution subshift $X$ and the 
2\,:\,1 morphism $\sigma$ from Example 4.3 of \cite{BSTY19} as cited above in Remark \ref{after-referee} (1), which has all the required properties.
\end{enumerate}
\end{enumerate}
\end{rem}


\section{The injectivity of the measure transfer for letter-to-letter morphisms}
\label{sec:6d}

Throughout this section we assume that $\sigma: \cal A^* \to \cal B^*$ is a letter-to-letter morphism of free monoids, and that $X \subset \cal A^\Z$ is a given subshift. From Definition-Remark \ref{3.5} we observe directly that in this special case the induced measure transfer map $\sigma^\cal M$ is given by the classical push-forward map $\mu \mapsto \sigma_*(\mu)$.

\begin{lem}
\label{8n.1}
If $\sigma$ is shift-orbit injective in $X$, and if for some ${\bf x} \in X$ the image word $\sigma({\bf x})$ is periodic, then $\bf x$ is periodic too.
\end{lem}

\begin{proof}
From the definition of the image subhift (see Definition-Remark \ref{image-shift} 
and Lemma \ref{closure-not})
it follows directly 
that for any biinfinite word ${\bf x} \in X$ the closure $\overline{\cal O({\bf x})}$ of its orbit $\cal O({\bf x})$ is a subshift which satisfies
$$\sigma(\overline{\cal O({\bf x})}) = \overline{\cal O(\sigma({\bf x}))} \, .$$
The claim is hence a direct consequence of the well known and easy to prove fact that the shift-orbit of any biinfinite word is closed if and only if the word is periodic.
\end{proof}

For any word $w \in \cal L(X)$ and any integer $n \geq 0$ let us consider the set $W_n(w)$ of both-sided prolongations of $w$ in $\cal L(X)$ by $n$ letters on each side, i.e.
\begin{equation}
\label{eq8a.1}
W_n(w) \,\, := \,\, \{ u w v  \in \cal L(X) \mid |u| = |v| = n \}\, .
\end{equation}
For any $w' = u w v \in W_n(w)$ we
now consider the image $\sigma(w')$ and its preimage set 
$\sigma^{-1}(\sigma(w')) \cap \cal L(X)$. We split the set $W_n(w)$ into a disjoint union,
\begin{equation}
\label{eq8a.2}
W_n(w) = U_n(w) 
\sqcup A_n(w) \, ,
\end{equation}
by setting
\begin{equation}
\label{eq8a.3}
w' \in U_n(w)  \,\, \Longleftrightarrow \,\, \sigma^{-1}(\sigma(w')) \cap \cal L(X) \subset W_n(w)
\end{equation}
and
\begin{equation}
\label{eq8a.4}
w' \in A_n(w)  \,\, \Longleftrightarrow \,\, \sigma^{-1}(\sigma(w')) \cap \cal L(X) \nsubseteq W_n(w) \, .
\end{equation}
In other words, for any $w' = u w v \in A_n(w)$ there exists a word $u' w'' v' \in \cal L(X)$ with $|u'|= |v'| = n$ and $|w''| = |w|$ such that $w'' \neq w$ and $\sigma(u' w'' v') = \sigma(u w v)$, 
while for any 
$w' = u w v \in U_n(w)$ any word $u' w'' v' \in \cal L(X)$ with $|u'|= |v'| = n$ and $\sigma(u' w'' v') = \sigma(u w v)$ one has $w'' = w$ (and thus $u' w'' v' \in W_n(w)$).

\medskip

We now fix some $w \in \cal L(X)$ and observe directly from the definition of the sets $A_n(w)$ 
that their elements define cylinders for which their union $[A_n(w)] := \bigcup\{[u w v] \mid uwv \in A_n(w)\} \subset X$ satisfies
$$T^{-n-1}([A_{n+1}(w)]) \,\,\subset \,\, T^{-n}([A_n(w)])\,.$$

\begin{rem}
\label{8n.1.5}
(1)
We denote by $A_\infty(w)$ the intersection of the nested family of the $T^{-n}([A_n(w)])$.
Since by definition we have $A_n(w) \subset \cal L(X)$ for all $n \geq 0$ we know that $A_\infty(w) \subset X$.

\smallskip
\noindent
(2)
For any invariant measure $\mu$ on $X$ and any $n \geq 0$ one has 
thus 
$\mu([A_{n+1}(w)]) \leq \mu([A_n(w)])$ and 
$$\lim \mu([A_n(w)]) = \mu(A_\infty(w)) \, .$$
\end{rem}

\begin{lem}
\label{8n.2}
If $\sigma$ is shift-orbit injective in $X$, then for any $w \in \cal L(X)$ the intersection $A_\infty(w)$ consists only of periodic words.
\end{lem}

\begin{proof}
For any biinfinite word ${\bf x} \in A_\infty(w)$ there exists, by definition of $A_\infty(w)$, a sequence of words $u_n, v_n, u'_n, v'_n, w'_n \in \cal L(X)$ which have the following properties:
\begin{enumerate}
\item
$|u_n| = |v_n| = |u'_n| = |v'_n| = n 
\quad \text{and } \quad |w'_n| = |w|$
\item
$u'_n w'_n v'_n \in \cal L(X)$ and $u_n w v_n \in \cal L(X)$
\item
$\sigma(u_n w v_n) = \sigma(u'_n w'_n v'_n)$
\item
$w'_n \neq w$
\item
$\lim w v_n = {\bf x}_{[1, +\infty)}$ and $\lim u_n = {\bf x}_{(-\infty , 0]}$ 
\end{enumerate}
The statement (5) needs a bit of interpretation: We think of the words $w v_n$ as being indexed from $1$ to $|w v_n|$ and of $u_n$ as being indexed from $-|u_n|+1$ to $0$, and we pass to the limit while keeping the indices fixed. In other words, statement (5) is equivalent to
\begin{enumerate}
\item[(5')]
$\lim {\bf y}(n) 
= {\bf x}$ in $\cal A^\Z$, where for any $n \geq 0$ the biinfinite word ${\bf y}(n)$ is defined through ${\bf y}_{[1, \infty)}(n) = w v_n^{+\infty}$ and ${\bf y}_{(\infty, 0]}(n) = u_n^{-\infty}$.
\end{enumerate}

We now proceed analogously with the sequence of words $w'_n v'_n$ (indexed from $1$ to $|w'_n v'_n|$) and the words $u'_n$ (indexed from $-|u'_n|+1$ to $0$),
and pass to a subsequence of the integers $n$ such that there exists a biinfinite ``limit word'' ${\bf z} \in \cal A^\Z$ which satisfies
\begin{enumerate}
\item[(6)]
$\lim w'_n v'_n = {\bf z}_{[1, +\infty)}$ and $\lim u'_n = {\bf z}_{(-\infty , 0]} \, .$
\end{enumerate}
Again, these limits are meant to be equivalent to the statement
\begin{enumerate}
\item[(6')]
$\lim {\bf z}(n) 
= {\bf z}$ in $\cal A^\Z$, where the ${\bf z}(n)  \in \cal A^\Z$ are defined via ${\bf z}_{[1, \infty)}(n) = w'_n {v'_n}^{+\infty}$ and ${\bf z}_{(\infty, 0]}(n) = {u'_n}^{-\infty}$.
\end{enumerate}
Since by property (2) we have $u'_n w'_n v'_n \in \cal L(X)$, for all indices $n$ in the limits (6) or (6'), we obtain ${\bf z} \in X$. Furthermore, from the above limit set-up we know that the factor ${\bf z}_{[1, |w|]}$ of $\bf z$ is equal to some $w'_n$, and thus (by property (4)) distinct from $w$. We thus deduce
$${\bf z} \neq {\bf x} \, .$$
On the other hand, we obtain from property (3) directly
$$\sigma({\bf z}) = \lim \sigma({\bf z}_n) = \lim \sigma({\bf y}_n) = \sigma({\bf x}) \, .$$
We now use the assumption that $\sigma$ is shift-orbit injective in $X$ to deduce that $\bf x$ and $\bf z$ are shift-translates of each other: There exists some integer $k \neq 0$ such that $T^k({\bf z}) = {\bf x}$. We thus have $T^k(\sigma({\bf x})) = T^k(\sigma({\bf z})) = \sigma(T^k({\bf z})) = \sigma({\bf x})$. But any biinfinite word which is equal to a non-trivial shift-translate of itself must be periodic. Hence $\sigma({\bf x})$ is periodic, and now we can employ Lemma \ref{8n.1} in order to deduce that the biinfinite word ${\bf x}$ is periodic.
\end{proof}

We denote by $\Per(X)$ the subset of all biinfinite periodic words in $X$. We observe that both, the countable set $\Per(X)$ and its complement $X \smallsetminus \Per(X)$ are measurable subsets of $\cal A^\Z$. 
Any measure $\mu \in \cal M(X)$ satisfies $\mu(\Per(X)) = 0$ if and only if $\mu$ is non-atomic. This gives rise to a 
decomposition
\begin{equation}
\label{eq8.5}
\mu \,\, = \,\, \mu^{per}  +  \mu^{na}
\end{equation}
of $\mu$ as sum of a non-atomic measure $\mu^{na}$ and a 
{\em periodic} measure $\mu^{per}$, 
by which we mean that $\mu^{per}$ is 
a countable or finite sum of scalar multiples of the characteristic measures $\mu_w$ from (\ref{eq:charac-m}).
Note that this decomposition is canonical, in that 
for any measurable set $B \subset X$ one has
\begin{equation}
\label{eq5n.unique}
\mu^{per}(B)  \,\, = \,\, \mu(B \cap \Per(X))  \quad \text{and} \quad  \mu^{na}(B) \,\,= \,\,\mu(B \cap (X \smallsetminus \Per(X))) \, .
\end{equation}

\begin{prop}
\label{8.3n}
Let $\mu$ and $\mu'$ be two invariant measures on $X$ 
which are both non-atomic.
Assume furthermore that $\sigma$ is shift-orbit injective.
Then one has
$$\sigma_*(\mu) = \sigma_*(\mu') \quad \Longrightarrow \quad\mu = \mu' \, .$$
\end{prop}

\begin{proof}
We fix an arbitrary word $w \in \cal L(X)$ and consider for any $n \geq 0$ the subsets $W_n(w), U_n(w)$ and $A_n(w)$ as defined in the equalities (\ref{eq8a.1}), (\ref{eq8a.3}) and (\ref{eq8a.4}), 
abbreviated here to $W_n, U_n$ and $A_n$ respectively.

Let us also fix some integer $n \geq 0$. From the definition of a weight function (see Section \ref{sec:2.1}) we see directly that for any measure $\mu$ on $X$ the corresponding weight function 
satisfies
\begin{equation}
\label{eq8a.5}
\mu(w) = \sum_{w' \in W_n(w)} \mu(w') \, .
\end{equation}
From (\ref{eq8a.2}) we know that 
the set $W_n$ decomposes as disjoint union of $U_n$ and $A_n$, 
so that (\ref{eq8a.5}) gives us
\begin{equation}
\label{eq7.8g}
\mu(w) = \mu(U_n) + \mu(A_n) \, ,
\end{equation}
where we set $\mu(U_n) := \underset{w' \in U_n}{\sum}\mu(w')$ and $\mu(A_n) := \underset{w' \in A_n}{\sum}\mu(w')$.

We now define $V_n := \sigma(U_n)$ and observe from the definition of $U_n$ that
one has 
$\sigma^{-1}(V_n) \cap \cal L(X) = U_n$. 
Hence we 
deduce from the definition of the push-forward measure $\sigma_*(\mu)$, 
together with the fact that by definition of $\mu \in \cal M(X)$ one has 
$\mu(w_0) = 0$ for any $w_0 \notin \cal L(X)$, 
that 
\begin{equation}
\label{eq7.9g}
\mu(U_n) 
= \sum_{v \in V_n} \sigma_*(\mu)(v) =: \sigma_*(\mu)(V_n) \, .
\end{equation}

We thus deduce from the equalities (\ref{eq7.8g}) and (\ref{eq7.9g}) that for any integer $n \geq 0$ one has
$$\mu(w) = \mu(U_n) + \mu(A_n) = \sigma_*(\mu)(V_n) + \mu(A_n) \, .$$
We now pass to the limit for $n \to \infty$ and 
recall from Remark \ref{8n.1.5} (2) 
that $\lim \mu(A_n) = \mu(A_\infty(w))$. From Lemma \ref{8n.2} we know $A_\infty(w) \subset \Per(X)$, so that our assumption that $\mu$ is 
non-atomic and hence 
zero on $\Per(X)$ implies $\lim \mu(A_n) = 0$. Hence 
$\lim \sigma_*(\mu)(V_n)$ does exist and satisfies
\begin{equation}
\label{eq7.10g}
\lim_{n \to \infty} \sigma_*(\mu)(V_n) = \mu(w) \, .
\end{equation}
Since the sets $V_n$ depend only on $w$ and not on $\mu$, equality 
(\ref{eq7.10g}) shows 
that $\mu(w)$ is entirely determined by the values of the induced measure $\sigma_*(\mu)$, which proves our claim.
\end{proof}

\begin{lem}
\label{8n.4}
Let $\mu$ be any shift-invariant measure on the subshift $X$. Then we have:
\begin{enumerate}
\item
If $\mu$ is periodic (i.e. $\mu(X \smallsetminus \Per(X) = 0$), then 
$\sigma_*(\mu)$ 
is also periodic (i.e. $\sigma_*(\mu)(\sigma(X) \smallsetminus \Per(\sigma(X)) = 0$).
\item
Under the additional hypothesis that $\sigma$ is shift-orbit injective, we also have: If $\mu$ is non-atomic (i.e. $\mu(\Per(X)) = 0$), then $\sigma_*(\mu)$ is also non-atomic (i.e. $\sigma_*(\mu)(\Per(\sigma(X))) = 0$).
\end{enumerate}
\end{lem}

\begin{proof}
(1)  If $\mu(X \smallsetminus \Per(X)) = 0$, then $\mu$ is carried entirely by the (countable) set of periodic words in $X$: There exist a coefficient $\lambda_w \geq 0$ for any $w \in \cal L(X)$ such that, using the characteristic measures $\mu_w$ from (\ref{eq:charac-m}), the measure $\mu$ can be expressed as countable sum
$$\mu = \sum_{w \in \cal L(X)} \lambda_w \,\mu_w \, .$$
But then Lemma \ref{3.6} (d) gives directly
$\sigma_*(\mu) = \underset{w \in \cal L(X)}{\sum} \lambda_w \,\mu_{\sigma(w)}\,$, 
which is clearly the zero-measure outside of $\Per(\sigma(X))$.

\smallskip
\noindent
(2)
For any ${\bf x} \in \Per(\sigma(X))$ we know from Lemma \ref{8n.1} that $\sigma^{-1}({\bf x})$ consist of periodic words only. Hence the assumption $\mu(\Per(X)) = 0$ implies that $\sigma_*(\mu)({\bf x}) = \mu(\sigma^{-1}({\bf x})) = 0$. Since $\Per(X)$ is countable, this implies $\sigma_*(\mu)(\Per(X)) = 0$.
\end{proof}

\begin{prop}
\label{8n.4.5}
For the subshift $X \subset \cal A^\Z$ let $\cal M_{per}(X) \subset \cal M(X)$ and  $\cal M_{na}(X) \subset \cal M(X)$ be the subsets of periodic and of non-atomic invariant measures on $X$ respectively.

For any shift-orbit injective morphism $\sigma: \cal A^* \to \cal B^*$ the injectivity of the push-forward map 
$\sigma_*^X: \cal M(X) \to \cal M(\sigma(X)), \, \mu \mapsto \sigma_*(\mu)$ follows if one proves the injectivity for the restrictions of $\sigma_*^X$ to both, $\cal M_{per}(X)$ and $\cal M_{na}(X)$.
\end{prop}

\begin{proof}
Let $\mu_1, \mu_2 \in \cal M(X)$ be invariant measures which satisfy $\sigma_*(\mu_1) = \sigma_*(\mu_1) =: \mu_0$. Consider the canonical decompositions into a periodic and an non-atomic measure from (\ref{eq8.5}) given by $\mu_1 = \mu_1^{per} + \mu_1^{na}$, $\mu_2 = \mu_2^{per} + \mu_2^{na}$ and $\mu_0 = \mu_0^{per} + \mu_0^{na}$. From Lemma \ref{8n.4} 
and the uniqueness of the canonical decomposition for $\mu_0$ (see the equalities (\ref{eq5n.unique})) we derive 
that $\sigma_*(\mu_1^{per}) = \sigma_*(\mu_2^{per}) = \mu_0^{per}$ and $\sigma_*(\mu_1^{na}) = \sigma_*(\mu_2^{na}) = \mu_0^{na}$. Injectivity of $\sigma_*^X$ on $\cal M_{per}(X)$ and $\cal M_{na}(X)$ implies $\mu_1^{per} = \mu_2^{per}$ and $\mu_1^{na} = \mu_2^{na}$, which shows $\mu_1 = \mu_2$.
\end{proof}

\begin{thm}
\label{8n.5}
For any letter-to-letter morphism $\sigma: \cal A^* \to \cal B^*$ and any subshift $X \subset \cal A^\Z$ one has:  If $\sigma$ is shift-orbit injective in $X$, then the push-forward map $\cal M(X) \to \cal M(\sigma(X)), \, \mu \mapsto \sigma_*(\mu)$ is injective.
\end{thm}

\begin{proof}
From Proposition \ref{8n.4.5} we know that it suffices to prove the injectivity of the measure map $\mu \mapsto \sigma_*(\mu)$ for the two cases, where 
$\mu$ is periodic or where $\mu$ is non-atomic.

The latter case is already dealt with in Proposition \ref{8.3n}, so that we can from now on assume that $\mu$ is carried by $\Per(X)$. But in this case the injectivity of the map $\mu \mapsto \sigma_*(\mu)$ is trivial, since by the injectivity assumption on the shift-orbits together with Lemma \ref{8n.1} we know that for every periodic word ${\bf x} \in \sigma(X)$ the set $\sigma^{-1}({\bf x}) \cap X$ 
consists of words ${\bf y}_i$ which are periodic as well and belong all to a single shift-orbit $\cal O({\bf y}_i)$. Hence we have $\mu(\cal O({\bf y}_i)) = \sigma_*(\mu)(\cal O({\bf x}))$. 
Since any invariant measure on $X$ which is carried by $\Per(X)$ is determined by knowing its evaluation on every periodic orbit, it follows that 
$\mu$ is entirely determined by $\sigma_*(\mu)$.
\end{proof}

The following remark serves as continuation of the discussion started at the end of last section:

\begin{rem}
\label{8n.7}
(1)
We would like to point out a subtlety in the last proof:  Although the preimage set $\sigma^{-1}(\cal O({\bf x}))$ of any periodic orbit $\cal O({\bf x})$ is a single periodic orbit, it is in general not true that the preimage set $\sigma^{-1}({\bf x})$ consists of a single biinfinite word ${\bf y}_i \in \sigma^{-1}(\cal O({\bf x}))$. It is here that the missing assumption ``shift-period preserving'' materializes, compared to the case where we assume ``$\sigma$ recognizable in $X$'' instead of our weaker assumption ``$\sigma$ is shift-orbit injective''.

\smallskip
\noindent
(2)
On the other hand, as pointed out at the end of last section, ``shift-orbit injective'' is a slightly stronger property than ``recognizable for aperiodic points'', and as we have observed in the last section, the latter doesn't suffice to deduce the injectivity of the push-forward map on $\cal M(X)$.
In fact, this injectivity 
fails in general already for non-atomic measures on
the given subshift $X$, as can be seen easily from the morphism which sends all letters of an alphabet of size $\geq 2$ to a single letter of the image alphabet; any such morphism is a forteriori recognizable for aperiodic points in the full shift
(see Warning \ref{Warning}).
\end{rem}

We terminate this section by 
considering some concrete examples how Theorem \ref{8n.5} or rather Theorem \ref{5n.3} can be put to use. An immediate application, for $\sigma, X$ and $Y = \sigma(X)$ as above, would be to deduce from the assumption that $Y$ is uniquely ergodic the conclusion that $X$ is also uniquely ergodic. Below we exhibit something slightly more elaborate:

\begin{example}
\label{6.9new}
(1)
Let us first consider for $\cal A = \{a, b\}$ and $\cal B = \{c, d\}$ the morphisms $\sigma_i: \cal A^* \to \cal B^*$ (for $i = 1$ and $i = 2$) given by
$$\sigma_1(a) = c^2\, , \,\, \sigma_1(b) = d \quad \text{and} \quad \sigma_2(a) = cdc\, , \,\, \sigma_2(b) = dcd \, .$$
Neither of these morphisms is recognizable in the full shift, since $c^{\pm \infty}$ admits two different desubstitutions with respect to $\sigma_1\,$, and $(cd)^{\pm \infty}$ admits three different desubstitutions with respect to $\sigma_2\,$. Nevertheless, it is easy to verify by elementary arguments that both, $\sigma_1$ and $\sigma_2\,$, are 1-1 on the set of shift-orbits of $\cal A^\Z$.

\smallskip
\noindent
(2)
Consider now any subshift $X \subset \cal A^\Z$ with image $Y_i = \sigma_i(X)$, and assume that $X$ admits $k \geq 2$ distinct ergodic probability measures
$\mu_j$. Then the $\mu_j^{\sigma_i}$ give $k$ distinct ergodic measures on $Y_i\,$, and after normalization we obtain $k$ distinct ergodic probability measures on $Y_i\,$, which are the only ones on $Y_i$ (by the surjectivity of the measure transfer map $\sigma^\cal M$).

A non-evident case occurs if  $X$ is one of the subshifts exhibited in Theorem 7.4 of \cite{BHL2.8-II} which (i) are minimal, (ii) have topological entropy $h_X = 0$, and (iii) possess infinitely many pairwise different ergodic probability measures. It follows that the $Y_i$ also have all of the properties (i), (ii) and (iii).

\smallskip
\noindent
(3)
There are many direct generalizations of the morphisms $\sigma_i$ as given above, for instance those given for any $n \geq 2$ and $1 \leq k \leq n$ by
$$\sigma'_1(a_k) = a_k^{j_k} \,\, (j_k \geq 1) \quad \text{and} \quad 
\sigma'_2(a_k) = a_k a_{k+1} \ldots a_n 
a_1 a_2 \ldots a_{k-1} 
a_k \, ,$$
or by suitable compositions of those (and of morphisms which are recognizable in the full shift). 
For each of these composed morphisms 
the above observation yields subshifts $Y_i$ as in (2) above.
This has lead us to pose the following question:

{\em 
Does every 
infinite transitive sofic subshift 
contain a subshift with all the properties (i), (ii) and (iii) from part (2) above ?}
\end{example}

\end{document}